\documentclass[12pt]{amsart}
\usepackage{tikz-cd}
\usepackage{graphicx}
\usepackage{amssymb}
\usepackage{tikz}
\usetikzlibrary{matrix,arrows.meta,bending}
\usepackage{color}
\usepackage[numbers]{natbib}
\usetikzlibrary{arrows.meta,positioning,fit,calc}
\usepackage[a4paper, left=3cm, right=3cm, top=3cm, bottom=3cm, footskip=15mm]{geometry}

\newtheorem{theorem}{Theorem}[section]

\newtheorem{lemma}[theorem]{Lemma}
\newtheorem{proposition}[theorem]{Proposition}
\newtheorem{corollary}[theorem]{Corollary}

\theoremstyle{definition}
\newtheorem{definition}[theorem]{Definition}

\theoremstyle{remark}

\newtheorem{question}[theorem]{Question}
\numberwithin{equation}{section}

\makeindex

\usepackage{fancyhdr}
\usepackage{calc} % ensures dimension arithmetic works robustly

\AtBeginDocument{%
  \setlength{\textwidth}{\dimexpr\paperwidth - 6cm\relax}%
  \setlength{\oddsidemargin}{\dimexpr 3cm - 1in\relax}%
  \setlength{\evensidemargin}{\dimexpr 3cm - 1in\relax}%
  \setlength{\marginparwidth}{2.5cm}%
}

\begin{document}

\title{The Compression method and applications}

%    Information for first author
\author{T. Agama}
    %Address of record for the research reported here
\address{Department of Mathematics, African Institute for Mathematical science, Ghana
}
%    Current address
%\curraddr{Department of Mathematics and Statistics,
%{Case Western Reserve University, Cleveland, Ohio 43403}
\email{theophilus@aims.edu.gh/emperordagama@yahoo.com}
%    \thanks will become a 1st page footnote.
%\thanks{The first author was supported in part by NSF Grant \#000000.}

    %Information for second author
%\author{Author Two}
%\address{Department of Mathematics, African Institute for Mathematical science, Ghana
%}
%\email{Gael@aims.edu.gh}
%\thanks{Support information for the second author.}

%    General info

\date{\today}

%\dedicatory{}


\begin{abstract}
In this paper, we introduce and develop the method of compression of points in space. We introduce the notion of the mass, the rank, the entropy, the cover and the energy of compression. We leverage this method to prove some class of inequalities related to Diophantine equations. In particular, we show that for each $L<n-1$ and for each $K>n-1$, there exist some $(x_1,x_2,\ldots,x_n)\in \mathbb{N}^n$ with $x_i\neq x_j$ for all $1\leq i<j\leq n$ such that 
\begin{align}
\frac{1}{K^{n}}\ll \prod \limits_{j=1}^{n}\frac{1}{x_j}\ll \frac{\log (\frac{n}{L})}{nL^{n-1}}\nonumber
\end{align}
and that for each $L>n-1$ there exist some $(x_1,x_2,\ldots,x_n)$ with $x_i\neq x_j$ for all $1\leq i<j\leq n$ and some $s\geq 2$ such that 
\begin{align}
\sum \limits_{j=1}^{n}\frac{1}{x_j^s}\gg s\frac{n}{L^{s-1}}.\nonumber
\end{align}
\end{abstract}

\maketitle

\begingroup
  \setlength{\parskip}{6pt} % <--- change this number to taste
  \tableofcontents
\endgroup

\section{Introduction}

The study of reciprocal structures in number theory and geometry has produced a number of classical problems whose formulations are deceptively simple but whose answers remain highly nontrivial. Among the most prominent are the Erd\H{o}s--Straus conjecture on unit fractions, the unit-distance problem in Euclidean geometry, the Gauss circle problem, and the Ehrhart volume conjecture in the geometry of numbers. The Erd\H{o}s--Straus equation asks whether every integer $n\ge 2$ can be represented as a sum of three unit fractions; this problem has generated a substantial literature and remains a central object in the theory of Egyptian fractions \cite{elsholtz2013counting}. In discrete geometry, the unit-distance problem asks for extremal bounds on the number of pairs of points at a prescribed distance, and the planar case has been a driving force behind several important developments in combinatorial geometry \cite{spencer1984unit}. In the geometry of numbers, the Gauss circle problem concerns the discrepancy between lattice point counts and Euclidean volume, whereas the Ehrhart conjecture predicts a sharp upper bound for the volume of a convex body whose barycenter is its only interior lattice point \cite{huxley2002integer,ehrhart1979volume,berman2017volume}.\\

Motivated by these themes, we introduce a pointwise reciprocal rescaling, which we call \emph{compression}, and develop the associated geometric invariants of the construction. Starting from a point $\vec{x}=(x_1,\dots,x_n)$ with nonzero coordinates, the compression map sends $\vec{x}$ to a scaled reciprocal vector, and from this we define the quantities that organize the paper: mass, rank, entropy, compression gap, energy, measure, and cost. The guiding principle is that compression turns arithmetic data into a geometric configuration centered at the midpoint between a point and its compressed image. In particular, the compression gap is the Euclidean distance between a point and its image under compression, whereas the induced ball is the ball centered at the midpoint of the segment joining these two points and having radius one half of that gap. This geometry provides a unified language for formulating inequalities, comparing scales, and isolating special configurations of points.\\

The main purpose of the paper is to systematically develop this framework and to record several consequences that connect compression geometry with classical problems in number theory and discrete geometry. The results presented here are intended to show that the compression formalism can be used to package a range of estimates into a single geometric picture. The method is especially natural when studying tuples of distinct positive coordinates, where the reciprocal transformation and associated ball geometry can be exploited to produce bounds for unit sums, higher-power sums, and related counting questions.\\

\subsection{Organization of the paper} The paper is organized as follows. Section~2 develops the basic compression map and introduces the main invariants attached to it, including the mass, rank, entropy, compression gap, energy, measure, and cost. Section~3 studies the geometry of compression-induced balls, including admissible points, interior points, limit points, and nesting phenomena. Section~4 presents applications of the framework to questions in discrete geometry, the geometry of numbers, and related counting problems. The final section contains further remarks and comments on possible directions for future work.

\section{Compression}

\begin{definition}
By the \emph{compression} of the fixed scale $m$ with $1\geq m>0$ on the points in $\mathbb{R}^{n}$, we mean the map 
$$
\mathbb{V}_m:(\mathbb{R}\setminus\{0\})^n\longrightarrow \mathbb{R}^n
$$ 
such that 
\begin{align}
\mathbb{V}_m[(x_1,x_2,\ldots, x_n)]=\bigg(\frac{m}{x_1},\frac{m}{x_2},\ldots, \frac{m}{x_n}\bigg)\nonumber
\end{align}
for $n\geq 2$ and with $x_i\neq 0$ for $1\leq i\leq n$. 
\end{definition}
\bigskip

% Preamble:
% \usepackage{tikz}
% \usetikzlibrary{calc,arrows.meta,decorations.pathreplacing}

\begin{figure}[ht]
\centering
\begin{tikzpicture}[>=Latex,scale=1.05]

% Coordinates for the point and its compressed image
\coordinate (X) at (1.2,2.6);
\coordinate (Y) at (5.6,1.0);
\coordinate (M) at ($(X)!0.5!(Y)$);

% Axes
\draw[->,thick] (-0.3,0) -- (7.1,0) node[right] {$x$};
\draw[->,thick] (0,-0.3) -- (0,4.0) node[above] {$y$};

% Original point and compressed image
\filldraw[black] (X) circle (2pt);
\filldraw[black] (Y) circle (2pt);

\node[above left] at (X) {$\vec{x}$};
\node[below right] at (Y) {$\mathbb{V}_m(\vec{x})$};

% Segment between point and its image
\draw[very thick,blue!70] (X) -- (Y);

% Midpoint
\filldraw[red!80!black] (M) circle (2pt);
\node[below] at (M) {midpoint};

% Induced ball
\draw[red!70,thick] (M) circle (1.60);

% Radius label
\draw[<->,red!70] ($(M)+(0,1.60)$) -- node[right] {$\frac12\,\mathcal{G}\circ\mathbb{V}_m[\vec{x}]$} ($(M)+(0,0)$);

% Compression arrow
\draw[->,very thick,blue!80] ($(X)+(0.15,0.15)$) to[bend left=18]
node[above,sloped] {compression} ($(Y)+(-0.15,-0.10)$);

% Text labels
\node at (3.5,3.55) {\small Compression of a point};
\node at (3.5,3.20) {\small $\mathbb{V}_m(\vec{x})=\left(\frac{m}{x_1},\frac{m}{x_2}\right)$};
\node at (3.5,2.85) {\small $\mathcal{G}\circ\mathbb{V}_m[\vec{x}]=\left\|\vec{x}-\mathbb{V}_m(\vec{x})\right\|$};

% Optional coordinate annotations
\node[below left] at (X) {\small $(x_1,x_2)$};
\node[above right] at (Y) {\small $\left(\frac{m}{x_1},\frac{m}{x_2}\right)$};

\end{tikzpicture}
\caption{A schematic depiction of compression of a point and the induced ball.}
\end{figure}

The notion of compression is the process of rescaling points in $\mathbb{R}^n$ for $n\geq 2$. We observe that a compression pushes points very close to the origin away from the origin by a certain scale and similarly draws points away from the origin close to the origin. Intuitively, one could think of a compression as inducing a certain kind of motion on points in the Euclidean space of any dimension. 

\begin{proposition}
A compression $\mathbb{V}_m:(\mathbb{R}\setminus \{0\})^n\longrightarrow \mathbb{R}^n$ of fixed scale $m$ with $1\geq m>0$ is a bijective map.
\end{proposition}

\begin{proof}
Suppose that $\mathbb{V}_m[(x_1,x_2,\ldots, x_n)]=\mathbb{V}_m[(y_1,y_2,\ldots,y_n)]$. We get
\begin{align}
\bigg(\frac{m}{x_1},\frac{m}{x_2},\ldots,\frac{m}{x_n}\bigg)=\bigg(\frac{m}{y_1},\frac{m}{y_2},\ldots,\frac{m}{y_n}\bigg).\nonumber
\end{align}
We deduce $x_i=y_i$ for each $i=1,2,\ldots, n$. Surjectivity follows by the definition of the map. Thus, the map is bijective.
\end{proof}

\subsection{The mass of compression}

\begin{definition}\label{mass}
By the \emph{mass} of a compression $\mathbb{V}_m$ of fixed scale $m$ with $0<m\leq 1$, we mean the map $\mathcal{M}:\mathbb{R}^n\longrightarrow \mathbb{R}$ such that \begin{align}
\mathcal{M}(\mathbb{V}_m[(x_1,x_2,\ldots,x_n)])=\sum \limits_{i=1}^{n}\frac{m}{x_i}.\nonumber
\end{align}
\end{definition}
\bigskip

Here, we deduce an upper and lower bound for the mass of the compression of scale $m$ with $0<m\leq 1$.

\begin{proposition}\label{crucial}
Let $(x_1,x_2,\ldots,x_n)\in \mathbb{R}^n$ with $x_i\neq x_j$ for each $i\neq j$. We have
\begin{align}
m\log \bigg(1-\frac{n-1}{\mathrm{sup}(x_j)}\bigg)^{-1}\ll \mathcal{M}(\mathbb{V}_m[(x_1,x_2,\ldots, x_n)])\ll m\log \bigg(1+\frac{n-1}{\mathrm{inf}(x_j)}\bigg)\nonumber
\end{align}
for $n\geq 2$.
\end{proposition}

\begin{proof}
Let $(x_1,x_2,\ldots,x_n)\in \mathbb{R}^n$ for $n\geq 2$ with $x_j\geq 1$. We get 
\begin{align}
\mathcal{M}(\mathbb{V}_m[(x_1,x_2,\ldots, x_n)])&=m\sum \limits_{j=1}^{n}\frac{1}{x_j}\nonumber \\&\leq m\sum \limits_{k=0}^{n-1}\frac{1}{\mathrm{inf}(x_j)+k}\nonumber
\end{align}
and deduce the upper bound. The lower bound is obtained by observing that
\begin{align}
\mathcal{M}(\mathbb{V}_m[(x_1,x_2,\ldots, x_n)])&=m\sum \limits_{j=1}^{n}\frac{1}{x_j}\nonumber \\&\geq m\sum \limits_{k=0}^{n-1}\frac{1}{\mathrm{sup}(x_j)-k}.\nonumber
\end{align}
\end{proof}
\bigskip

The bounds obtained for the mass of compression are quite suggestive. It restricts the choice of the entries to be distinct. After a little heuristics, we observe that the lower bound for the mass of compression will be flawed if we allow for tuples with at least two similar entries. Thus in building this framework, and with all the results we will obtained, we will enforce that the entries of any choice of tuple are distinct.

\subsubsection{Application of the mass of compression}

In this section, we apply the notion of the mass of compression to the Erd\'{o}s-Straus conjecture. 

\begin{theorem}\label{weak edos}
There exists some $(x_1,x_2,\ldots,x_n)\in \mathbb{N}^n$ for each $n\geq 2$ with $x_j\geq 1$ such that 
\begin{align}
m\frac{n}{L_1}\ll \mathcal{M}(\mathbb{V}_m[(x_1,x_2,\ldots, x_n)])\ll m\frac{n}{L_2}\nonumber
\end{align}
for some $L_1,L_2\in \mathbb{N}$.
\end{theorem}

\begin{proof}
We choose $(x_1,x_2,\ldots,x_n)\in \mathbb{N}^{n}$ such that $\mathrm{sup}(x_j)>\mathrm{inf}(x_j)>n-1$ for $j=1,\ldots n$. By Proposition \ref{crucial}, we have the upper bound 
\begin{align}
\mathcal{M}(\mathbb{V}_m[(x_1,x_2,\ldots, x_n)])&\ll m\log \bigg(1+\frac{n-1}{\mathrm{inf}(x_j)}\bigg)\nonumber \\&= m\sum \limits_{k=1}^{\infty}\frac{(-1)^{k-1}}{k}\bigg(\frac{n-1}{\mathrm{inf}(x_j)}\bigg)^k\nonumber \\& \ll m\frac{n}{\mathrm{inf}(x_j)}.\nonumber
\end{align}
The lower bound is obtained by observing
\begin{align}
\mathcal{M}(\mathbb{V}_m[(x_1,x_2,\ldots, x_n)])&\gg m\log \bigg(1-\frac{n-1}{\mathrm{sup}(x_j)}\bigg)^{-1}\nonumber \\&=m\sum \limits_{k=1}^{\infty}\frac{1}{k}\bigg(\frac{n-1}{\mathrm{sup}(x_j)}\bigg)^{k}\nonumber \\& \gg m\frac{n}{\mathrm{sup}(x_j)}.\nonumber
\end{align}
The claimed inequality is deduced by choosing $\mathrm{sup}(x_j)=L_1$ and $\mathrm{inf}(x_j)=L_2$.
\end{proof}
\bigskip

Theorem \ref{weak edos} is redolent of the Ed\`{o}s-Strauss conjecture. It can be considered a weaker version of the conjecture. It is implicit from Theorem \ref{weak edos} that there are infinitely many points in $\mathbb{N}^n$ that satisfy the inequality with finitely many such exceptions. Therefore, in the opposite direction, we can conclude that there are infinitely many $L_1,L_2\in \mathbb{N}$ satisfying the inequality. We state a consequence of the result in Theorem \ref{weak edos} to shed light on this assertion.

\begin{corollary}
For each $L\in \mathbb{N}$ with $L>n-1$, there exist some $(x_1,x_2,\ldots,x_n)\in \mathbb{N}^n$ with $x_i\neq x_j$ for all $1\leq i<j\leq n$ such that 
\begin{align}
\frac{n}{L}\ll \sum \limits_{j=1}^{n}\frac{1}{x_j}\ll \frac{n}{L}\nonumber
\end{align}
In particular, for each $L\geq 3$ there exists some $(x_1,x_2,x_3)\in \mathbb{N}^3$ with $x_1\neq x_2$, $x_2\neq x_3$ and $x_1\neq x_3$ such that 
\begin{align}
\frac{3}{L}\ll \frac{1}{x_1}+\frac{1}{x_2}+\frac{1}{x_3}\ll \frac{3}{L}.\nonumber
\end{align}
\end{corollary}

\begin{proof}
First, we choose $(x_1,x_2,\ldots,x_n)\in \mathbb{N}^n$ with $x_i\neq x_j$ for all $1\leq i<j\leq n$ such that $\mathrm{sup}(x_j)>\mathrm{inf}(x_j)>n-1$. Taking $K=\mathrm{sup}(x_j)$ and $L=\mathrm{inf}(x_j)$ for any such points, we get
\begin{align}
\frac{n}{L}\ll \sum\limits_{j=1}^{n}\frac{1}{x_j}\ll \frac{n}{K}\ll \frac{n}{L}.\nonumber
\end{align}
The special case follows by taking $n=3$.
\end{proof}
\bigskip

The result can be interpreted as saying that for each $L\geq 3$ there exist some $(x_1,x_2,x_3)\in \mathbb{N}^3$ such that \begin{align}
c_1\frac{3}{L}\leq \frac{1}{x_1}+\frac{1}{x_2}+\frac{1}{x_3}\leq c_2\frac{3}{L}\nonumber
\end{align}
for some constants $c_1,c_2>1$. The Erd\'{o}s-Straus conjecture will follow when we take $c_1=c_2=\frac{4}{3}$. Investigating the scale of these constants is partly the motivation for this framework.

\begin{theorem}\label{massapp2}
For each $K>n-1$ and for each $L<n-1$, there exists some $(x_1,x_2,\ldots,x_n)\in \mathbb{N}^n$ with $x_i\neq x_j$ for all $1\leq i<j\leq n$ such that 
\begin{align}
\frac{n}{K}\ll \sum\limits_{j=1}^{n}\frac{1}{x_j}\ll \log \bigg(\frac{n}{L}\bigg).\nonumber
\end{align}
\end{theorem}

\begin{proof}
We choose $(x_1,x_2,\ldots,x_n)\in \mathbb{N}^n$ with $x_i\neq x_j$ for all $1\leq i<j\leq n$ such that $\mathrm{inf}(x_j)<n-1$ and $\mathrm{sup}(x_j)>n-1$. Consequently, we set $L=\mathrm{inf}(x_j)$ and $K=\mathrm{sup}(x_j)$ and deduce the inequality in Theorem \ref{crucial}.
\end{proof}

\begin{corollary}
For each $K>2$, there exists some $(x_1,x_2,x_3)\in \mathbb{N}^3$ with $x_i\neq x_j$ for all $1\leq i<j\leq 3$ such that \begin{align}
c_1\frac{3}{K}\leq \frac{1}{x_1}+\frac{1}{x_2}+\frac{1}{x_3}\leq c_2\log 3\nonumber
\end{align}
for some $c_1,c_2>1$.
\end{corollary}

\subsection{The rank of compression}

In this section, we introduce the notion of the \emph{rank} of compression.

\begin{definition}\label{rank}
Let $(x_1,x_2,\ldots,x_n)\in \mathbb{R}^n$ for $n\geq 2$. By the \emph{rank} of compression $\mathbb{V}_m$, we mean the metric
\begin{align}
\mathcal{R}\circ \mathbb{V}_m[(x_1,x_2,\ldots, x_n)]=\bigg|\bigg|\bigg(\frac{m}{x_1},\frac{m}{x_2},\ldots,\frac{m}{x_n}\bigg)\bigg|\bigg|.\nonumber
\end{align}
\end{definition}
\bigskip

The rank of a compression $\mathbb{V}_m$ of fixed scale $m$ with $1\geq m>0$ is fundamentally the distance of the image of the points under compression from the origin. Here, we relate the rank of compression of a fixed scale $m$ with the mass of a certain compression of scale $m:=1$.

\begin{proposition}\label{rank-mass}
Let $(x_1,x_2,\ldots, x_n)\in \mathbb{R}^n$. We have
\begin{align}
\mathcal{R}\circ \mathbb{V}_m[(x_1,x_2,\ldots, x_n)]^2=m^2\mathcal{M}\circ\mathbb{V}_1\left [\bigg(x_1^2,x_2^2,\ldots,x_n^2\bigg)\right].\nonumber
\end{align}
\end{proposition}

\begin{proof}
The result follows from definition \ref{rank} and definition \ref{mass}.
\end{proof}

\begin{theorem}\label{ranktheorem}
Let $(x_1,x_2,\ldots,x_n)\in \mathbb{N}^n$. We have 
\begin{align}
m\sqrt{\log \bigg(1-\frac{n-1}{\mathrm{sup}(x_j^2)}\bigg)^{-1}}\ll\mathcal{R}\circ \mathbb{V}_m[(x_1,x_2,\ldots, x_n)]\ll m\sqrt{\log \bigg(1+\frac{n-1}{\mathrm{inf}(x_j^2)}\bigg)}\nonumber
\end{align}
\end{theorem}

\begin{proof}
The inequality is deduced from Proposition \ref{rank-mass} and Proposition \ref{crucial}.
\end{proof}

\subsubsection{Application of rank of compression}

In this section, we obtain one consequence of the rank of compression. We apply these bounds to the second-moment unit sums of the Erd\'{o}s-type problem.

\begin{theorem}
For each $L>\sqrt{n-1}$, there exists some $(x_1,x_2,\ldots,x_n)\in \mathbb{N}^n$ with $x_i\neq x_j$ for all $1\leq i<j\leq n$ such that 
\begin{align}
\frac{n}{L^2}\ll \sum \limits_{j=1}^{n}\frac{1}{x_j^2}\ll \frac{n}{L^2}.\nonumber
\end{align}
In particular, for each $L\geq 2$, there exists some $(x_1,x_2,x_3)\in \mathbb{N}^3$ with $x_1\neq x_2$, $x_2\neq x_3$ and $x_1\neq x_3$ and some $c_1,c_2>1$ such that 
\begin{align}
c_1\frac{3}{L^2}\leq \frac{1}{x_1^2}+\frac{1}{x_2^2}+\frac{1}{x_3^2}\leq c_2\frac{3}{L^2}.\nonumber
\end{align}
\end{theorem}

\begin{proof}
We choose $(x_1,x_2,\ldots,x_n)\in \mathbb{N}^n$ in Theorem \ref{ranktheorem} such that $L=\mathrm{inf}(x_j)$ with $L^2>n-1$. The inequality follows immediately. The special case follows by taking $n=3$.
\end{proof}

\begin{corollary}
For each $L\geq 3$, there exists some $(x_1,x_2,x_3,x_4,x_5)\in \mathbb{N}^5$ with $x_i\neq x_j$ for all $1\leq i<j\leq 5$ and some $c_1,c_2>1$ such that 
\begin{align}
c_1\frac{5}{L^2}\leq \frac{1}{x_1^2}+\frac{1}{x_2^2}+\frac{1}{x_3^2}+\frac{1}{x_4^2}+\frac{1}{x_5^2}\leq c_2\frac{5}{L^2}.\nonumber
\end{align}
\end{corollary}

\subsection{The entropy of compression}

In this section, we introduce the notion of the \emph{entropy} of compression. Intuitively, one could think of this concept as a measure that assigns a weight to the image of points under compression. We provide some modest bounds for this statistic and exploit some applications in the context of some Diophantine problems.

\begin{definition}\label{entropy}
Let $(x_1,x_2,\ldots,x_n)\in \mathbb{R}^n$ with $x_i\neq 0,1$ for all $i=1,2\ldots,n$. By the \emph{entropy} of a compression $\mathbb{V}_m$ of scale $1\geq m>0$, we mean the map $\mathcal{E}:\mathbb{R}^n\longrightarrow \mathbb{R}$ such that \begin{align}
\mathcal{E}(\mathbb{V}_m[(x_1,x_2,\ldots, x_n)])=\prod\limits_{i=1}^{n}\frac{m}{x_i}.\nonumber
\end{align}
\end{definition}
\bigskip

Here, we relate the mass of a compression to the entropy of compression and deduce reasonable good bounds for our further studies.

\begin{proposition}\label{massentropy}
For all $n\geq 2$, we have 
\begin{align}\mathcal{M}(\mathbb{V}_m[(x_1,x_2,\ldots,x_n)])=m\mathcal{M}\bigg(\mathbb{V}_1\bigg[\bigg(\prod \limits_{i\neq 1}\frac{1}{x_i},\prod \limits_{i\neq 2}\frac{1}{x_i},\ldots,\prod \limits_{i\neq n}\frac{1}{x_i}\bigg)\bigg]\bigg)\times \mathcal{E}(\mathbb{V}_1[(x_1,x_2,\ldots, x_n)]).\nonumber
\end{align}
\end{proposition}

\begin{proof}
By definition \ref{mass}, we have 
\begin{align}
\mathcal{M}(\mathbb{V}_m[(x_1,x_2,\ldots,x_n)])&=\sum \limits_{i=1}^{n}\frac{m}{x_i}\nonumber \\&=m\frac{\sum \limits_{\sigma:[1,n]\longrightarrow [1,n]}\prod \limits_{\substack{n-1\\\sigma(i)\neq \sigma(j)\\i\neq j\\i\in [1,n]}}x_{\sigma(i)}}{\prod \limits_{i=1}^{n}x_i}.\nonumber
\end{align}
We deduce the claim from this relation.
\end{proof}

\begin{proposition}\label{entropy1}
Let $(x_1,x_2,\ldots,x_n)\in\mathbb{N}^n$ with $x_i\neq x_j$ for $i\neq j$. We have 
\begin{align}
\frac{\log (1-\frac{n-1}{\mathrm{sup}(x_j)})^{-1}}{n\mathrm{sup}(x_j)^{n-1}}\ll \mathcal{E}(\mathbb{V}_1[(x_1,x_2,\ldots, x_n)])\ll \frac{\log (1+\frac{n-1}{\mathrm{inf}(x_j)})}{n\mathrm{inf}(x_j)^{n-1}}.\nonumber
\end{align}
\end{proposition}

\begin{proof}
The bounds are obtained using the relation in Proposition \ref{massentropy}, the bounds in Proposition \ref{crucial} and observing that 
\begin{align}
\mathcal{M}\bigg(\mathbb{V}_1\bigg[\bigg(\prod \limits_{i\neq 1}\frac{1}{x_i},\prod \limits_{i\neq 2}\frac{1}{x_i},\ldots, \prod \limits_{i\neq n}\frac{1}{x_i}\bigg)\bigg]\bigg)\leq n\mathrm{sup}(x_j)^{n-1}\nonumber
\end{align}
and 
\begin{align}
\mathcal{M}\bigg(\mathbb{V}_1\bigg[\bigg(\prod \limits_{i\neq 1}\frac{1}{x_i},\prod \limits_{i\neq 2}\frac{1}{x_i},\ldots, \prod \limits_{i\neq n}\frac{1}{x_i}\bigg)\bigg]\bigg)\geq n\mathrm{inf}(x_j)^{n-1}.\nonumber
\end{align}
\end{proof}

\subsubsection{Applications of the entropy of compression}

In this section, we deduce an important consequence of the entropy of compression. One could think of these applications as analogs of the Erd\'{o}s-type results for the unit sums of triples of the form $(x_1,x_2,x_3)$.

\begin{theorem}\label{entropyapplication1}
For  each $L>n-1$, there exists some $(x_1,x_2,\ldots,x_{n})\in \mathbb{N}^n$ with $x_i\neq x_j$ for all $1\leq i<j\leq n$ such that 
\begin{align}
\frac{1}{L^{n}}\ll \prod\limits_{i=1}^{n}\frac{1}{x_i}\ll \frac{1}{L^{n}}.\nonumber
\end{align}
\end{theorem}

\begin{proof}
We choose $(x_1,x_2,\ldots,x_n)\in \mathbb{N}^n$ with $x_i\neq x_j$ for all $1\leq i<j\leq n$ such that $L>n-1$ with $\mathrm{inf}(x_j)=L$. The result follows immediately in Proposition \ref{entropy}.
\end{proof}
\bigskip

Theorem \ref{entropyapplication1} suggests that for some tuple $(x_1,x_2,\ldots,x_n)\in \mathbb{N}^n$ with $x_i\neq x_j$ for all $1\leq i<j\leq n$ there must exist some $c_1,c_2>1$ such that \begin{align}
\frac{c_1}{L^{n}}\leq \prod \limits_{j=1}^{n}\frac{1}{x_j}\leq \frac{c_2}{L^{n}}.\nonumber
\end{align}

\begin{theorem}
For each $L<n-1$ and for each $K>n-1$, there exists some $(x_1,x_2,\ldots,x_n)\in \mathbb{N}^n$ with $x_i\neq x_j$ for all $1\leq i<j\leq n$ such that 
\begin{align}
\frac{1}{K^{n}}\ll \prod \limits_{j=1}^{n}\frac{1}{x_j}\ll \frac{\log (\frac{n}{L})}{nL^{n-1}}.\nonumber
\end{align}
\end{theorem}

\begin{proof}
We choose a tuple $(x_1,x_2,\ldots,x_n)\in \mathbb{N}^n$ with $x_i\neq x_j$ for all $1\leq i<j\leq n$ such that $\mathrm{sup}(x_j)=K>n-1$ and $L=\mathrm{inf}(x_j)<n-1$.
\end{proof}

\begin{corollary}
For each $L<4$ and for each $K>4$, there exists some $(x_1,x_2,x_3,x_4,x_5)\in \mathbb{N}^5$ with $x_i\neq x_j$ for  all $1\leq i<j\leq 5$ and some $c_1,c_2>1$ such that 
\begin{align}
\frac{c_1}{K^5}\leq \frac{1}{x_1}\times\frac{1}{x_2}\times \frac{1}{x_3}\times\frac{1}{x_4}\times \frac{1}{x_5}\leq c_2\frac{\log 5}{5L^4}.\nonumber
\end{align}
\end{corollary}

\begin{proof}
The result follows by taking $n=5$ in Theorem \ref{entropyapplication1}.
\end{proof}

\subsection{The compression gap}

In this section, we introduce the notion of the \emph{compression gap}. 

\begin{definition}\label{gap}
Let $(x_1,x_2,\ldots, x_n)\in \mathbb{R}^n$ with $x_i\neq 0,1$ for all $i=1,2\ldots,n$. By the \emph{compression gap} of compression $\mathbb{V}_m$ of scale $m$ with $1\geq m>0$, denoted by $\mathcal{G}\circ \mathbb{V}_m[(x_1,x_2,\ldots, x_n)]$, we mean the metric
\begin{align}
\mathcal{G}\circ \mathbb{V}_m[(x_1,x_2,\ldots, x_n)]=\bigg|\bigg|\bigg(x_1-\frac{m}{x_1},x_2-\frac{m}{x_2},\ldots,x_n-\frac{m}{x_n}\bigg)\bigg|\bigg|\nonumber
\end{align}
\end{definition}
\bigskip

The compression gap is a definitive measure of the chasm between the points and their image points under compression. We can bound this by relating the compression gap to the mass of an expansion in the following ways.

\begin{proposition}\label{cgidentity1}
Let $(x_1,x_2,\ldots, x_n)\in \mathbb{R}^n$ for $n\geq 2$ with $x_j\neq 0,1$ for $j=1,\ldots,n$. We have 
\begin{align}
\mathcal{G}\circ\mathbb{V}_m[(x_1,x_2,\ldots, x_n)]^2=\mathcal{M}\circ \mathbb{V}_1\bigg[\bigg(\frac{1}{x_1^2},\ldots,\frac{1}{x_n^2}\bigg)\bigg]+m^2\mathcal{M}\circ \mathbb{V}_1[(x_1^2,\ldots,x_n^2)]-2mn.\nonumber
\end{align}
\end{proposition}

\begin{proof}
This follows from definition \ref{gap} and definition \ref{mass}.
\end{proof}

\begin{theorem}\label{gap2}
Let $(x_1,x_2,\ldots, x_n)\in \mathbb{R}^n$ with $x_i\neq x_j$~($i\neq j$) for $n\geq 2$ and let $m:=m(n)=o(1)$ as $n\longrightarrow \infty$. We have 
\begin{align}
\mathcal{G}\circ\mathbb{V}_m[(x_1,x_2,\ldots, x_n)]^2\ll n\mathrm{sup}(x_j^2)+m^2\log\bigg(1+\frac{n-1}{\mathrm{inf}(x_j)^2}\bigg)-2mn\nonumber
\end{align} 
and 
\begin{align}
\mathcal{G}\circ\mathbb{V}_m[(x_1,x_2,\ldots, x_n)]^2\gg n\mathrm{inf}(x_j^2)+m^2\log\bigg(1-\frac{n-1}{\mathrm{sup}(x_j^2)}\bigg)^{-1}-2mn.\nonumber
\end{align}
\end{theorem}

\begin{proof}
The bounds are obtained by using Proposition \ref{crucial} in Proposition \ref{cgidentity1} and observing that 
\begin{align}
n\mathrm{inf}(x_j^2)\leq \mathcal{M}\circ \mathbb{V}_1\bigg[\bigg(\frac{1}{x_1^2},\ldots,\frac{1}{x_n^2}\bigg)\bigg]\leq n\mathrm{sup}(x_j^2).\nonumber
\end{align}
\end{proof}
\bigskip

\subsection{The energy of compression}

In this section, we introduce the notion of the \emph{energy} of compression.

\begin{definition}\label{energy1}
Let $(x_1,x_2,\ldots, x_n)\in \mathbb{R}^{n}$ with $x_i\neq 0,1$ for all $i=1,2\ldots,n$ for $n\geq 2$. By the \emph{energy} dissipated under compression on $(x_1,x_2,\ldots, x_n)$, denoted by $\mathbb{E}$, we mean the measure 
\begin{align}
\mathbb{E}\circ \mathbb{V}_m[(x_1,x_2,\ldots, x_n)]=\mathcal{G}\circ \mathbb{V}_m[(x_1,x_2,\ldots, x_n)]\times \mathcal{E}(\mathbb{V}_m\bigg[\bigg(x_1,x_2,\ldots, x_n\bigg)\bigg]).\nonumber
\end{align}
\end{definition}
\bigskip

With the upper and lower bounds for the compression gap and the entropy of any points under compression, we can similarly bound the energy dissipated under compression.

\begin{proposition}\label{energy2}
Let $(x_1,x_2,\ldots,x_n)\in\mathbb{R}^n$ with $x_i\neq x_j$ for each $i\neq j$ and $x_i\neq 0$ for $1\leq i\leq n$. If $m:=m(n)=o(1)$ as $n\longrightarrow \infty$, then we have 
\begin{align}
\mathbb{E}\circ\mathbb{V}_m[(x_1,x_2,\ldots, x_n)]&\ll \frac{\mathrm{sup}(x_j)}{(\mathrm{inf}(x_j))^{n-1}\sqrt{n}}\log \bigg(1+\frac{n-1}{\mathrm{inf}(x_j)}\bigg)\nonumber 
\end{align}
and 
\begin{align}
\mathbb{E}\circ\mathbb{V}_m[(x_1,x_2,\ldots, x_n)]&\gg \frac{\mathrm{inf}(x_j)}{\sqrt{n}(\mathrm{sup}(x_j))^{n-1}}\log \bigg(1-\frac{n-1}{\mathrm{sup}(x_j)}\bigg)^{-1}.\nonumber
\end{align}
\end{proposition}

\begin{proof}
The result follows by inserting the bounds in \ref{gap2} and \ref{entropy} into definition \ref{energy1}.
\end{proof}

\subsubsection{Applications of the energy of compression}

In this section, we give some applications of the energy dissipated by points under compression.

\begin{theorem}\label{energy3}
For each $K>n-1$ and for each $L<n-1$, there exists some $(x_1,x_2,\ldots,x_n)\in \mathbb{N}^n$ with $x_i\neq x_j$ for all $1\leq i<j\leq n$ such that 
\begin{align}
\frac{L}{K^{n-1}\sqrt{n}}\ll \frac{\bigg|\bigg|\bigg(x_1-\frac{1}{x_1},x_2-\frac{1}{x_2},\ldots,x_n-\frac{1}{x_n}\bigg)\bigg|\bigg|}{x_1x_2\cdots x_n}\ll \frac{K\log \bigg(\frac{n}{L}\bigg)}{L^{n-2}\sqrt{n}}.\nonumber
\end{align}
\end{theorem}

\begin{proof}
We choose $(x_1,x_2,\ldots,x_n)\in \mathbb{N}^n$ with $x_i\neq x_j$ for all $1\leq i<j\leq n$ such that $\mathrm{inf}(x_j)<n-1$ and $\mathrm{sup}(x_j)>n-1$ and set $K=\mathrm{sup}(x_j)$ and $\mathrm{inf}(x_j)=L$. The claimed bounds are deduced by exploiting the estimates in Proposition \ref{energy2}. 
\end{proof}

\begin{corollary}
For each $K\geq 5$ and for each $L<4$, there exists some $(x_1,x_2,x_3,x_4,x_5)\in\mathbb{N}^5$ with $x_i\neq x_j$ for all $1\leq i<j\leq 5$ such that 
\begin{align}
\frac{L}{K^4\sqrt{5}}\ll \frac{\bigg|\bigg|\bigg(x_1-\frac{1}{x_1},x_2-\frac{1}{x_2},x_3-\frac{1}{x_3},x_4-\frac{1}{x_4},x_5-\frac{1}{x_5}\bigg)\bigg|\bigg|}{x_1x_2\cdots x_5}\ll \frac{K\log \bigg(\frac{5}{L}\bigg)}{L^3\sqrt{5}}.\nonumber
\end{align}
\end{corollary}

\begin{proof}
The result follows by taking $n=5$ in Theorem \ref{energy3}.
\end{proof}

\subsection{The measure and cost of compression}

In this section, we introduce the notion of the \emph{measure} and the \emph{cost} of compression.

\begin{definition}\label{measurecost}
Let $(x_1,x_2,\ldots,x_n)\in \mathbb{R}^n$ with $x_i\neq 0,1$ for all $i=1,2\ldots,n$ for $n\geq 2$. By the \emph{compression measure} of $(x_1,x_2,\ldots,x_n)$ under compression, denoted by $\mathcal{N}$, we mean the expression 
\begin{align}
\mathcal{N}\circ \mathbb{V}_m[(x_1,x_2,\ldots, x_n)]=\bigg|\mathcal{E}(\mathbb{V}_m[(x_1,x_2,\ldots, x_n)])-\mathcal{E}(\mathbb{V}_m\bigg[\bigg(\frac{1}{x_1},\frac{1}{x_2},\ldots, \frac{1}{x_n}\bigg)\bigg]\bigg)\bigg|.\nonumber
\end{align}
The corresponding \emph{compression cost}, denoted by $\mathcal{C}$, is
\begin{align}
\mathcal{C}\circ \mathbb{V}_m[(x_1,x_2,\ldots, x_n)]=\mathcal{N}\circ \mathbb{V}_m[(x_1,x_2,\ldots, x_n)]\times \mathcal{G}\circ \mathbb{V}_m[(x_1,x_2,\ldots, x_n)].\nonumber
\end{align}
\end{definition}
\bigskip

\begin{proposition}
Let $(x_1,x_2,\ldots,x_n)\in \mathbb{N}^n$. We have
\begin{align}
\mathcal{N}\circ\mathbb{V}_1[(x_1,x_2,\ldots, x_n)]&\ll \mathrm{sup}(x_j)^n \nonumber
\end{align}
and 
\begin{align}
\mathcal{N}\circ\mathbb{V}_1[(x_1,x_2,\ldots, x_n)]&\gg \mathrm{inf}(x_j)^n.\nonumber
\end{align}
\end{proposition}

\begin{proof}
The inequality is obtained using the bounds in Theorem \ref{entropy} in definition \ref{measurecost}.
\end{proof}

\begin{proposition}
Let $(x_1,x_2,\ldots,x_n)\in \mathbb{R}^n$ with $x_i\neq x_j$~($i\neq j$) with $x_i\neq 0$ for $1\leq i\leq n$. If $m:=m(n)=o(1)$ as $n\longrightarrow \infty$, then we have \begin{align}
\mathcal{C}\circ\mathbb{V}_m[(x_1,x_2,\ldots, x_n)]&\ll \mathrm{sup}(x_j)^{n+1}\sqrt{n}\nonumber
\end{align} 
and 
\begin{align}
\mathcal{C}\circ\mathbb{V}_m[(x_1,x_2,\ldots, x_n)]&\gg \mathrm{inf}(x_j)^{n+1}\sqrt{n}.\nonumber
\end{align} 
\end{proposition}

\begin{proof}
These are consequences of the developed bounds.
\end{proof}
\bigskip

\section{The ball induced by compression}

In this section, we introduce the notion of the \emph{ball} induced by a point $(x_1,x_2,\ldots,x_n)\in\mathbb{R}^n$ under compression of a given scale.

\begin{definition}
Let $(x_1,x_2,\ldots,x_n)\in \mathbb{R}^n$ with $x_i\neq x_j$ for all $1\leq i<j\leq n$ and $x_i\neq 0$ for all $1\leq i\leq n$. By the ball induced by the point $(x_1,x_2,\ldots,x_n)\in \mathbb{R}^n$ under compression $\mathbb{V}_m$ of fixed scale $m$ with $1\geq m>0$, denoted by 
$$
\mathcal{B}_{\frac{1}{2}\mathcal{G}\circ \mathbb{V}_m[(x_1,x_2,\ldots, x_n)]}[(x_1,x_2,\ldots,x_n)],
$$ 
we mean the set of points $\vec{y}\in \mathbb{R}^n$ that satisfy the inequality
\begin{align}
\left|\left|\vec{y}-\frac{1}{2}\bigg(x_1+\frac{m}{x_1},x_2+\frac{m}{x_2},\ldots,x_n+\frac{m}{x_n}\bigg)\right|\right|<\frac{1}{2}\mathcal{G}\circ \mathbb{V}_m[(x_1,x_2,\ldots, x_n)].\nonumber
\end{align}
A point $\vec{z}=(z_1,z_2,\ldots,z_n)\in \mathcal{B}_{\frac{1}{2}\mathcal{G}\circ \mathbb{V}_m[(x_1,x_2,\ldots, x_n)]}[(x_1,x_2,\ldots,x_n)]$ if it satisfies the inequality.
\end{definition}
\bigskip

We shall prove a nested property of balls induced by points under compression. Essentially, we show that the balls induced by points in a specific ball should be smaller and covered the ball containing the point.
\bigskip

In the geometry of balls induced under compression of scale $m>0$, we will implicitly assume that 
$$
0<m\leq 1.
$$
We will choose to write the ball induced by the point $\vec{x}=(x_1,x_2,\ldots,x_n)$ under compression as 
\begin{align}
\mathcal{B}_{\frac{1}{2}\mathcal{G}\circ\mathbb{V}_m[\vec{x}]}[\vec{x}].\nonumber
\end{align}
We adopt this notation to save enough work space in many circumstances.

\begin{proposition}\label{cgidentity}
Let $(x_1,x_2,\ldots, x_n)\in \mathbb{R}^n$ for $n\geq 2$ with $x_j\neq 0$ for $j=1,\ldots,n$. We have 
\begin{align}
\mathcal{G}\circ\mathbb{V}_m[(x_1,x_2,\ldots, x_n)]^2=\mathcal{M}\circ \mathbb{V}_1\bigg[\bigg(\frac{1}{x_1^2},\ldots,\frac{1}{x_n^2}\bigg)\bigg]+m^2\mathcal{M}\circ \mathbb{V}_1[(x_1^2,\ldots,x_n^2)]-2mn.\nonumber
\end{align}
In particular, if $m=m(n)=o(1)$ as $n\longrightarrow \infty$, then 
\begin{align}
\mathcal{G}\circ\mathbb{V}_m[(x_1,x_2,\ldots, x_n)]^2=\mathcal{M}\circ \mathbb{V}_1\bigg[\bigg(\frac{1}{x_1^2},\ldots,\frac{1}{x_n^2}\bigg)\bigg]-2mn+O\bigg(m^2\mathcal{M}\circ \mathbb{V}_1[(x_1^2,\ldots,x_n^2)]\bigg)\nonumber
\end{align}
for $\vec{x}\in \mathbb{R}^n$ with $x_i\geq 1$ for each $1\leq i\leq n.$
\end{proposition}
\bigskip

The proposition \ref{cgidentity} offers an extremely useful identity. It allows a pass from the compression gap on points to their relative distance to the origin. It suggests that points under compression with a large gap must be far away from the origin than points with a relatively smaller gap under compression. That is, the inequality 
\begin{align}
\mathcal{G}\circ \mathbb{V}_m[\vec{x}]<\mathcal{G}\circ \mathbb{V}_m[\vec{y}]\nonumber
\end{align}
holds with $m:=m(n)=o(1)$ as $n\longrightarrow \infty$ if and only if $||\vec{x}||\lesssim ||\vec{y}||$ for $\vec{x}, \vec{y}\in \mathbb{R}^n$ with $x_i\geq 1$ for all $1\leq i\leq n$. This important transference principle will be exploited in most cases. In particular, we note that in the latter case, we can write the asymptotic 
$$
\mathcal{G}\circ\mathbb{V}_m[(x_1,x_2,\ldots, x_n)]^2\sim \mathcal{M}\circ\mathbb{V}_1\bigg[\bigg(\frac{1}{x_1^2},\ldots,\frac{1}{x_n^2}\bigg)\bigg]=||\vec{x}||^2.
$$

\begin{lemma}[Compression estimate]\label{gapestimate}
Let $(x_1,x_2,\ldots, x_n)\in\mathbb{R}^n$ for $n\geq 2$ with $x_i\geq 1$ for all $1\leq i\leq n$ with $x_i\neq x_j$~($i\neq j$). If $m:=m(n)=o(1)$ as $n\longrightarrow \infty$, then
\begin{align}
\mathcal{G}\circ \mathbb{V}_m[(x_1,x_2,\ldots, x_n)]^2\ll n\mathrm{sup}(x_j^2)+m^2\log \bigg(1+\frac{n-1}{\mathrm{inf}(x_j)^2}\bigg)-2mn\nonumber
\end{align} 
and 
\begin{align}
\mathcal{G}\circ \mathbb{V}_m[(x_1,x_2,\ldots, x_n)]^2\gg n\mathrm{inf}(x_j^2)+m^2\log \bigg(1-\frac{n-1}{\mathrm{sup}(x_j^2)}\bigg)^{-1}-2mn.\nonumber
\end{align}
\end{lemma}
\bigskip

\begin{theorem}\label{decider}
Let $\vec{z}=(z_1,z_2,\ldots,z_n)\in \mathbb{R}^n$ with $z_i\neq z_j$ for all $1\leq i<j\leq n$ with $y_i,z_i\geq 1$ for all $1\leq i\leq n$ and $m:=m(n)=o(1)$ as $n\longrightarrow \infty$. Then  $\vec{z}\in \mathcal{B}_{\frac{1}{2}\mathcal{G}\circ \mathbb{V}_m[\vec{y}]}[\vec{y}]$ with $||\vec{z}||<||\vec{y}||$ if and only if 
\begin{align}
\mathcal{G}\circ \mathbb{V}_m[\vec{z}]\leq \mathcal{G}\circ \mathbb{V}_m[\vec{y}]\nonumber
\end{align}
with $||\vec{y}-\vec{z}||<\epsilon$ for some $\epsilon>0$
\end{theorem}

\begin{proof}
Let $\vec{z}\in\mathcal{B}_{\frac{1}{2}\mathcal{G}\circ \mathbb{V}_m[\vec{y}]}[\vec{y}]$ for $\vec{z}=(z_1,z_2,\ldots,z_n)\in \mathbb{R}^n$ with $z_i\neq z_j$ for all $1\leq i<j\leq n$ and $z_i\geq 1$ for all $1\leq i\leq n$ such that $||\vec{y}||>||\vec{z}||$. Suppose that 
\begin{align}
\mathcal{G}\circ\mathbb{V}_m[\vec{z}]>\mathcal{G}\circ \mathbb{V}_m[\vec{y}].\nonumber
\end{align}
We deduce $||\vec{y}||\lesssim||\vec{z}||$, which is absurd. In this case, we can take $\epsilon:=\frac{1}{2}\mathcal{G}\circ \mathbb{V}_m[\vec{y}]$. Conversely, suppose that 
\begin{align}
\mathcal{G}\circ \mathbb{V}_m[\vec{z}]\leq \mathcal{G}\circ \mathbb{V}_m[\vec{y}].\nonumber
\end{align}
The proposition \ref{cgidentity} implies $||\vec{z}||\lesssim||\vec{y}||$. Under the requirement $||\vec{y}-\vec{z}||<\epsilon$ for some $\epsilon>0$, we obtain the inequality
\begin{align}
\bigg|\bigg|\vec{z}-\frac{1}{2}\bigg(y_1+\frac{m}{y_1},\ldots,y_n+\frac{m}{y_n}\bigg)\bigg|\bigg|&\leq \bigg|\bigg|\vec{y}-\frac{1}{2}\bigg(y_1+\frac{m}{y_1},\ldots,y_n+\frac{m}{y_n}\bigg)\bigg|\bigg|+\epsilon\nonumber \\&=\frac{1}{2}\mathcal{G}\circ \mathbb{V}_m[\vec{y}]+\epsilon \nonumber
\end{align}
with $m=m(n)=o(1)$ as $n\longrightarrow \infty$. Choosing $\epsilon>0$ sufficiently small, we deduce $\vec{z}\in \mathcal{B}_{\frac{1}{2}\mathcal{G}\circ\mathbb{V}_m[\vec{y}]}[\vec{y}]$ and the proof of the theorem is complete.
\end{proof}
\bigskip

In the geometry of balls under compression, we will assume that $n$ is sufficiently large for $\mathbb{R}^n$. In this regime, we will always take the scale of compression $m:=m(n)=o(1)$ as $n\longrightarrow \infty.$

\begin{theorem}\label{ballproof}
Let $\vec{x}=(x_1,x_2,\ldots,x_n)\in \mathbb{R}^n$ with $x_i\neq x_j$ for all $1\leq i<j\leq n$ with $y_i,x_i\geq 1$ for each $1\leq i\leq n$. If $\vec{y}\in \mathcal{B}_{\frac{1}{2}\mathcal{G}\circ \mathbb{V}_m[\vec{x}]}[\vec{x}]$ with $||\vec{y}||<||\vec{x}||$ for $||\vec{y}-\vec{x}||<\delta$ for sufficiently small $\delta>0$, then 
\begin{align}
\mathcal{B}_{\frac{1}{2}\mathcal{G}\circ\mathbb{V}_m[\vec{y}]}[\vec{y}]\subseteq \mathcal{B}_{\frac{1}{2}\mathcal{G}\circ \mathbb{V}_m[\vec{x}]}[\vec{x}]\nonumber
\end{align}
for $m:=m(n)=o(1)$ as $n\longrightarrow \infty.$
\end{theorem}

\begin{proof}
We let $\vec{y}\in \mathcal{B}_{\frac{1}{2}\mathcal{G}\circ \mathbb{V}_m[\vec{x}]}[\vec{x}]$ such that $y_i\geq 1$ for each $1\leq i\leq n$ with $||\vec{y}||<||\vec{x}||$ for $||\vec{y}-\vec{x}||<\delta$. Theorem \ref{decider} implies $\mathcal{G}\circ \mathbb{V}_m[\vec{x}]\gtrsim \mathcal{G}\circ \mathbb{V}_m[\vec{y}]$ with $||\vec{y}-\vec{x}||<\delta$ for sufficiently small $\delta>0$. Consequently, the ball $\mathcal{B}_{\frac{1}{2}\mathcal{G}\circ \mathbb{V}_m[\vec{x}]}[\vec{x}]$ is slightly larger than the ball $\mathcal{B}_{\frac{1}{2}\mathcal{G}\circ \mathbb{V}_m[\vec{y}]}[\vec{y}]$ due to their compression gaps, and the latter does not contain the point $\vec{x}$ by construction. We deduce $||\mathbb{V}_m[\vec{y}]||>||\mathbb{V}_m[\vec{x}]||$ and 
\begin{align}
\mathcal{G}\circ \mathbb{V}_m[\mathbb{V}_m[\vec{y}]]&=\mathcal{G}\circ \mathbb{V}_m[\vec{y}]\nonumber \\&\lesssim \mathcal{G}\circ \mathbb{V}_m[\vec{x}]\nonumber \\&=\mathcal{G}\circ \mathbb{V}_m[\mathbb{V}_m[\vec{x}]]\nonumber
\end{align}
with $||\mathbb{V}_m[\vec{y}]-\mathbb{V}_m[\vec{x}]||<\epsilon$ for small $\epsilon>0$. This implies that 
\begin{align}
\mathcal{B}_{\frac{1}{2}\mathcal{G}\circ \mathbb{V}_m[\vec{y}]}[\vec{y}]\subseteq \mathcal{B}_{\frac{1}{2}\mathcal{G}\circ \mathbb{V}_m[\vec{x}]}[\vec{x}]\nonumber
\end{align}
completing the proof. 
\end{proof}
\bigskip

Theorem \ref{ballproof} suggests that points confined to certain balls induced under compression must necessarily have their induced ball under compression contained in the larger ball.

\subsection{Interior points and the limit points of balls induced under compression}

In this section, we introduce the notion of an \emph{interior} and the \emph{limit point} of balls induced under compression.

\begin{definition}
Let $\vec{y}=(y_1,y_2,\ldots,y_n)\in\mathbb{R}^n$ with $y_i\neq y_j$ for all $1\leq i<j\leq n$. The point $\vec{z}\in \mathcal{B}_{\frac{1}{2}\mathcal{G}\circ\mathbb{V}_m[\vec{y}]}[\vec{y}]$ is an \emph{interior point} if 
\begin{align}
\mathcal{B}_{\frac{1}{2}\mathcal{G}\circ\mathbb{V}_m[\vec{z}]}[\vec{z}]\subseteq \mathcal{B}_{\frac{1}{2}\mathcal{G}\circ \mathbb{V}_m[\vec{x}]}[\vec{x}]\nonumber
\end{align}
for most $\vec{x}\in\mathcal{B}_{\frac{1}{2}\mathcal{G}\circ \mathbb{V}_m[\vec{y}]}[\vec{y}]$. An interior point $\vec{z}$ is then said to be a \emph{limit point} if 
\begin{align}
\mathcal{B}_{\frac{1}{2}\mathcal{G}\circ\mathbb{V}_m[\vec{z}]}[\vec{z}]\subseteq \mathcal{B}_{\frac{1}{2}\mathcal{G}\circ \mathbb{V}_m[\vec{x}]}[\vec{x}]\nonumber
\end{align}
for all $\vec{x}\in\mathcal{B}_{\frac{1}{2}\mathcal{G}\circ\mathbb{V}_m[\vec{y}]}[\vec{y}]$
\end{definition}
\bigskip

Here, we prove that there must exist an interior and limit point in any ball induced by points under compression of any scale in any dimension.

\begin{theorem}\label{limitexistence}
Let $\vec{x}=(x_1,x_2,\ldots,x_n)\in\mathbb{R}^n$ with $x_i\neq x_j$ for all $1\leq i<j\leq n$ with $y_i\geq 1$ for all $1\leq i\leq n$. The ball $\mathcal{B}_{\frac{1}{2}\mathcal{G}\circ \mathbb{V}_m[\vec{x}]}[\vec{x}]$ contains an interior point and a limit point.
\end{theorem}

\begin{proof}
Let $\vec{x}=(x_1,x_2,\ldots,x_n)\in \mathbb{R}^n$ with $x_i\neq x_j$ for all $1\leq i<j\leq n$ with $x_i\geq 1$ for all $1\leq i\leq n$ and suppose on the contrary that $\mathcal{B}_{\frac{1}{2}\mathcal{G}\circ \mathbb{V}_m[\vec{x}]}[\vec{x}]$ contains no limit point. We choose 
\begin{align}
\vec{z}_1\in \mathcal{B}_{\frac{1}{2}\mathcal{G}\circ \mathbb{V}_m[\vec{x}]}[\vec{x}]\nonumber
\end{align}
with $z_{1_i}\geq 1$ for each $1\leq i\leq n$ and $||\vec{z}_1||<||\vec{x}||$ such that $||\vec{z}_1-\vec{x}||<\epsilon$ for sufficiently small $\epsilon>0$. By Theorem \ref{ballproof} and Theorem \ref{decider}, it follows that 
\begin{align}
\mathcal{B}_{\frac{1}{2}\mathcal{G}\circ\mathbb{V}_m[\vec{z}_1]}[\vec{z}_1]\subset \mathcal{B}_{\frac{1}{2}\mathcal{G}\circ \mathbb{V}_m[\vec{x}]}[\vec{x}]\nonumber
\end{align}
with $\mathcal{G}\circ\mathbb{V}_m[\vec{z}_1]\lesssim \mathcal{G}\circ \mathbb{V}_m[\vec{x}]$. Again, we pick $\vec{z}_2\in \mathcal{B}_{\frac{1}{2}\mathcal{G}\circ \mathbb{V}_m[\vec{z}_1]}[\vec{z}_1]$ with $z_{2_i}\geq 1$ for each $1\leq i\leq n$ and $||\vec{z}_2||<||\vec{z}_1||$ such that  $||\vec{z}_2-\vec{z}_1||<\delta$ for sufficiently small $\delta>0$. By Theorem \ref{ballproof} and Theorem \ref{decider}, we have
\begin{align}
\mathcal{B}_{\frac{1}{2}\mathcal{G}\circ\mathbb{V}_m[\vec{z}_2]}[\vec{z}_2]\subset\mathcal{B}_{\frac{1}{2}\mathcal{G}\circ \mathbb{V}_m[\vec{z}_1]}[\vec{z}_1]\nonumber
\end{align}
with $\mathcal{G}\circ\mathbb{V}_m[\vec{z}_2]\lesssim \mathcal{G}\circ\mathbb{V}_m[\vec{z}_1]$. Continuing the argument in this manner, we obtain the infinite descending sequence of compression gaps
\begin{align}
\mathcal{G}\circ\mathbb{V}_m[\vec{x}]\gtrsim \mathcal{G}\circ \mathbb{V}_m[\vec{z}_1]\gtrsim \mathcal{G}\circ \mathbb{V}_m[\vec{z}_2]>\cdots \gtrsim \mathcal{G}\circ \mathbb{V}_m[\vec{z}_n]\gtrsim \cdots\nonumber
\end{align}
completing the proof.
\end{proof}

\begin{proposition}\label{found limit point}
The point $\vec{x}=(x_1,x_2,\ldots,x_n)$ with $x_i=1$ for each $1\leq i\leq n$ is the limit point of the ball $\mathcal{B}_{\frac{1}{2}\mathcal{G}\circ\mathbb{V}_1[\vec{y}]}[\vec{y}]$ for any $\vec{y}=(y_1,y_2,\ldots,y_n)\in \mathbb{R}^n$ with $y_i>1$ for each $1\leq i\leq n$.
\end{proposition}

\begin{proof}
Applying the compression $\mathbb{V}_1:\mathbb{R}^n\longrightarrow \mathbb{R}^n$ on the point $\vec{x}=(x_1,x_2,\ldots,x_n)$ with $x_i=1$ for each $1\leq i\leq n$, we obtain $\mathbb{V}_1[\vec{x}]=(1,1,\ldots,1)$ so that $\mathcal{G}\circ \mathbb{V}_1[\vec{x}]=0$ and the corresponding ball induced under compression $\mathcal{B}_{\frac{1}{2}\mathcal{G}\circ\mathbb{V}_1[\vec{x}]}[\vec{x}]$ contain only the point $\vec{x}$. By definition \ref{limitexistence}, the point $\vec{x}$ must be the limit point of the ball $\mathcal{B}_{\frac{1}{2}\mathcal{G}\circ \mathbb{V}_1[\vec{x}]}[\vec{x}]$. We deduce
\begin{align}
 \mathcal{B}_{\frac{1}{2}\mathcal{G}\circ\mathbb{V}_1[\vec{x}]}[\vec{x}]\subseteq \mathcal{B}_{\frac{1}{2}\mathcal{G}\circ \mathbb{V}_1[\vec{y}]}[\vec{y}]\nonumber 
\end{align}
for any $\vec{y}=(y_1,y_2,\ldots,y_n)\in \mathbb{R}^n$ with $y_i>1$ for all $1\leq i\leq n$. Assume to the contrary that 
\begin{align}
 \mathcal{B}_{\frac{1}{2}\mathcal{G}\circ\mathbb{V}_1[\vec{x}]}[\vec{x}]\not\subseteq \mathcal{B}_{\frac{1}{2}\mathcal{G}\circ \mathbb{V}_1[\vec{y}]}[\vec{y}]\nonumber    
\end{align}
for some $\vec{y}=(y_1,y_2,\ldots,y_n)\in \mathbb{R}^n$ with $y_i>1$ for each $1\leq i\leq n$. It implies that there exists some point $\vec{z}\in\mathcal{B}_{\frac{1}{2}\mathcal{G}\circ \mathbb{V}_1[\vec{x}]}[\vec{x}]$ such that $\vec{z}\not\in \mathcal{B}_{\frac{1}{2}\mathcal{G}\circ\mathbb{V}_1[\vec{y}]}[\vec{y}]$. Since $\vec{x}$ is the only point in the ball $\mathcal{B}_{\frac{1}{2}\mathcal{G}\circ\mathbb{V}_1[\vec{x}]}[\vec{x}]$, we deduce 
\begin{align}
    \vec{x}\not\in\mathcal{B}_{\frac{1}{2}\mathcal{G}\circ \mathbb{V}_1[\vec{y}]}[\vec{y}]\nonumber
\end{align}
which is inconsistent with the fact that $\vec{x}$ is the limit point of the ball.
\end{proof}

\subsection{Admissible points of balls induced under compression}

We introduce the notion of \emph{admissible} points of balls induced by points under compression.

\begin{definition}
Let $\vec{y}=(y_1,y_2,\ldots,y_n)\in\mathbb{R}^n$ with $y_i\neq y_j$ for all $1\leq i<j\leq n$. The point $\vec{y}$ is said to be an \emph{admissible} point of the ball $\mathcal{B}_{\frac{1}{2}\mathcal{G}\circ\mathbb{V}_m[\vec{x}]}[\vec{x}]$ if \begin{align}
\bigg|\bigg|\vec{y}-\frac{1}{2}\bigg(x_1+\frac{m}{x_1},\ldots,x_n+\frac{m}{x_n}\bigg)\bigg|\bigg|=\frac{1}{2}\mathcal{G}\circ\mathbb{V}_m[\vec{x}].\nonumber
\end{align}
\end{definition}
\bigskip

The admissible points of balls induced by points under compression encompass points on the ball. Geometrically, they sit on the outer surface of the induced ball. We will show that in principle all the balls can be generated by their admissible points.

\begin{theorem}\label{admissibletheorem}
Let $\vec{x}\in\mathbb{R}^n$ with $x_i\neq x_j$~($i\neq j$) such that $x_i\geq 1$ for all $1\leq i\leq n$ and set $m:=m(n)=o(1)$ as $n\longrightarrow \infty$. The point $\vec{y}\in \mathcal{B}_{\frac{1}{2}\mathcal{G}\circ\mathbb{V}_m[\vec{x}]}[\vec{x}]$ for $y_i\geq 1$ for each $1\leq i\leq n$ with $||\vec{y}||<||\vec{x}||$ such that $||\vec{y}-\vec{x}||<\epsilon$ for sufficiently small $\epsilon>0$ is admissible if and only if 
\begin{align}
\mathcal{B}_{\frac{1}{2}\mathcal{G}\circ\mathbb{V}_m[\vec{y}]}[\vec{y}]=\mathcal{B}_{\frac{1}{2}\mathcal{G}\circ \mathbb{V}_m[\vec{x}]}[\vec{x}]\nonumber
\end{align}
and $\mathcal{G}\circ\mathbb{V}_m[\vec{y}]=\mathcal{G}\circ \mathbb{V}_m[\vec{x}]$.
\end{theorem}

\begin{proof}
We let $\vec{y}\in \mathcal{B}_{\frac{1}{2}\mathcal{G}\circ \mathbb{V}_m[\vec{x}]}[\vec{x}]$ with $||\vec{y}||<||\vec{x}||$ such that $||\vec{y}-\vec{x}||<\epsilon$ for sufficiently small $\epsilon>0$ be admissible and suppose on the contrary that 
\begin{align}
\mathcal{B}_{\frac{1}{2}\mathcal{G}\circ \mathbb{V}_m[\vec{y}]}[\vec{y}]\neq \mathcal{B}_{\frac{1}{2}\mathcal{G}\circ \mathbb{V}_m[\vec{x}]}[\vec{x}].\nonumber
\end{align}
Without loss of generality, we can choose some $\vec{z}\in \mathcal{B}_{\frac{1}{2}\mathcal{G}\circ\mathbb{V}_m[\vec{x}]}[\vec{x}]$ with $||\vec{z}||<||\vec{x}||$ such that 
\begin{align}
\vec{z}\notin\mathcal{B}_{\frac{1}{2}\mathcal{G}\circ \mathbb{V}_m[\vec{y}]}[\vec{y}].\nonumber
\end{align}
such that $||\vec{z}-\vec{x}||<\delta$ for  sufficiently small $\delta>0$. Applying Theorem \ref{decider}, we obtain
\begin{align}
\mathcal{G}\circ \mathbb{V}_m[\vec{y}]\lesssim \mathcal{G}\circ \mathbb{V}_m[\vec{x}].\nonumber
\end{align}
This contradicts the equality $\mathcal{G}\circ \mathbb{V}_m[\vec{y}]=\mathcal{G}\circ \mathbb{V}_m[\vec{x}]$.  The latter equality of compression gaps follows from the requirement that the balls are indistinguishable. Conversely, suppose that
\begin{align}
\mathcal{B}_{\frac{1}{2}\mathcal{G}\circ\mathbb{V}_m[\vec{y}]}[\vec{y}]=\mathcal{B}_{\frac{1}{2}\mathcal{G}\circ \mathbb{V}_m[\vec{x}]}[\vec{x}]\nonumber
\end{align}
and $\mathcal{G}\circ \mathbb{V}_m[\vec{y}]=\mathcal{G}\circ \mathbb{V}_m[\vec{x}]$. It follows that the point $\vec{y}$ lives on the outer surface of the two indistinguishable balls and so must satisfy the equality
\begin{align}
\bigg|\bigg|\vec{z}-\frac{1}{2}\bigg(y_1+\frac{m}{y_1},\ldots,y_n+\frac{m}{y_n}\bigg)\bigg|\bigg|&=\bigg|\bigg|\vec{z}-\frac{1}{2}\bigg(x_1+\frac{m}{x_1},\ldots,x_n+\frac{m}{x_n}\bigg)\bigg|\bigg|\nonumber \\&=\frac{1}{2}\mathcal{G}\circ \mathbb{V}_m[\vec{x}].\nonumber
\end{align}
We deduce
\begin{align}
\frac{1}{2}\mathcal{G}\circ \mathbb{V}_m[\vec{x}]&=\bigg|\bigg|\vec{y}-\frac{1}{2}\bigg(x_1+\frac{m}{x_1},\ldots,x_n+\frac{m}{x_n}\bigg)\bigg|\bigg|\nonumber
\end{align}
and $\vec{y}$ is admissible.
\end{proof}
\bigskip

\section{Applications}
In this section, we apply the method developed to investigate some problems in discrete geometry.

\subsection{Application to the Erd\H{o}s unit distance problem}

Erd\H{o}s posed in 1946 the problem of counting the number of unit distances that can be determined by a set of $n$ points in the plane. It is known (see \cite{erdos1946sets}) that the number of unit distances that can be determined by $n$ points in the plane is lower bounded by 
\begin{align}
n^{1+\frac{c}{\log \log n}}.\nonumber
\end{align}
Erd\H{o}s asks if the upper bound for the number of unit distances that can be determined by $n$ points in the plane can also be a function of this form. In other words, the problem asks if the lower bound of Erd\H{o}s is the best possible. What is currently known is the upper bound (see \cite{spencer1984unit}) proportional to the quantity
\begin{align}
n^{\frac{4}{3}}\nonumber
\end{align}
due to Spencer, Szemer\H{e}di and Trotter.

\begin{definition}[Translation of balls]\label{translation of balls}
Let $\vec{x}\in\mathbb{R}^k$ and $\mathcal{B}_{\frac{1}{2}\mathcal{G}\circ\mathbb{V}_m[\vec{x}]}[\vec{x}]$ be the ball induced under compression. We denote the map
\begin{align}
\mathbb{T}_{\vec{v}}:\mathcal{B}_{\frac{1}{2}\mathcal{G}\circ \mathbb{V}_m[\vec{x}]}[\vec{x}]\longrightarrow \mathcal{B}^{\vec{v}}_{\frac{1}{2}\mathcal{G}\circ \mathbb{V}_m[\vec{x}]}[\vec{x}]\nonumber
\end{align}
as the translation of the ball by the vector $\vec{v}\in \mathbb{R}^k$, so that for any $\vec{y}\in \mathcal{B}_{\frac{1}{2}\mathcal{G}\circ \mathbb{V}_m[\vec{x}]}[\vec{x}]$ then 
\begin{align}
\vec{y}+\vec{v}\in\mathcal{B}^{\vec{v}}_{\frac{1}{2}\mathcal{G}\circ\mathbb{V}_m[\vec{x}]}[\vec{x}].\nonumber
\end{align}
\end{definition}

\begin{theorem}\label{weakerdos}
Let $\mathbb{E}\subset \mathbb{R}^2$ be a set of $n$ points in a general position, and 
$$
\mathcal{I}_n=\bigg\{(\vec{x}_t,\vec{x_j})\in \mathbb{E}\subset\mathbb{R}^2~:~||\vec{x_j}-\vec{x_t}||=1,~1\leq t,j\leq n\bigg\}.
$$ 
For any small $\epsilon>0$, we have 
\begin{align}
n^{1+\frac{1}{(\log n)^{1+\epsilon}}}\ll_\epsilon \#\mathcal{I}_n\ll_\epsilon n^{1+\frac{1}{(\log\log n)^{\epsilon}}}.\nonumber
\end{align}
\end{theorem}

\begin{proof}
We pick a point $\vec{x}_j \in \mathbb{R}^2$, set $\mathcal{G}\circ\mathbb{V}_m[\vec{x}_j]=1$ with $m:=m(2)=\frac{1}{2}$ and apply compression $\mathbb{V}_m$ to the point $\vec{x}_j$. We construct the ball induced under compression 
\begin{align}
\mathcal{B}_{\frac{1}{2}\mathcal{G}\circ\mathbb{V}_m[\vec{x}_j]}[\vec{x}_j].\nonumber
\end{align}
We note that the constructed ball has radius $\frac{1}{2}\mathcal{G}\circ\mathbb{V}_m[\vec{x}_j]=\frac{1}{2}$, so that for any admissible point $\vec{x}_k\neq \vec{x}_j$ of the ball $\mathcal{B}_{\frac{1}{2}\mathcal{G}\circ \mathbb{V}_m[\vec{x}_j]}[\vec{x}_j]$ there exists an admissible point $\vec{x}_l$ such that 
\begin{align}
||\vec{x}_k-\vec{x}_l||=1\nonumber
\end{align}
so that any such $\frac{n}{2}$ pairs of admissible points determines at least $\frac{n}{2}$ unit distances. Now for any $n$ such admissible points on the ball and by virtue of the restriction 
\begin{align}
\mathcal{G}\circ \mathbb{V}_m[\vec{x}_j]=1\label{unit compression gap}
\end{align}
we make the optimal assignment 
\begin{align}
\mathrm{min}_{1\leq j\leq n}\mathrm{inf}_{1\leq s\leq 2}(x_{j_s})=n^{\frac{1}{(\log n)^{1+\epsilon}}}\nonumber
\end{align}
and
\begin{align}
    \mathrm{max}_{1\leq j\leq n}\mathrm{sup}_{1\leq s\leq 2}(x_{j_s})<n^{\frac{1}{(\log\log n)^{\epsilon}}},\nonumber
\end{align}
for any small $\epsilon>0$, since points $\vec{x}_l$ far from the origin with $x_{l_s}$ for $1\leq s\leq 2$ must have large compression gaps due to the lemma \ref{gapestimate} and the ensuing discussion. In particular, the point $\vec{x}_l$ must be such that $x_{l_s}=1+\epsilon$ with $1\leq s\leq 2$ for any small $\epsilon>0$ in order to satisfy the requirement in \eqref{unit compression gap}. We deduce
\begin{align}
\sum \limits_{\substack{1\leq j\leq n\\x_j\in \mathbb{R}^2\\\mathcal{G}\circ\mathbb{V}_1[\vec{x}_j]=1}}1&\geq \sum \limits_{\substack{1\leq j\leq n\\ x_j\in \mathcal{B}_{\frac{1}{2}\mathcal{G}\circ\mathbb{V}_1[\vec{x}_j]}[\vec{x}_j]\bigcap \mathbb{R}^2\\\mathrm{min}_{1\leq j\leq n}\mathrm{inf}_{1\leq s\leq 2}(x_{j_s})=n^{\frac{1}{(\log n)^{1+\epsilon}}}}}\mathcal{G}\circ \mathbb{V}_1[\vec{x}_j]\nonumber \\& \gg_\epsilon \sum \limits_{\substack{1\leq j\leq n \\ \mathrm{min}_{1\leq j\leq n}\mathrm{inf}_{1\leq s\leq 2}(x_{j_s})=n^{\frac{1}{(\log n)^{1+\epsilon}}}}}\mathrm{inf}_{1\leq s\leq 2}(x_{j_s})\nonumber \\& \gg_\epsilon \sum \limits_{\substack{1\leq j\leq n \\ \mathrm{min}_{1\leq j\leq n}\mathrm{inf}_{1\leq s\leq 2}(x_{j_s})=n^{\frac{1}{(\log n)^{1+\epsilon}}}}}\mathrm{min}_{1\leq j\leq n}\mathrm{inf}_{1\leq s\leq 2}(x_{j_s})\nonumber \\&=n^{\frac{1}{(\log n)^{1+\epsilon}}}\sum \limits_{1\leq j\leq n}1 \nonumber \\&= n^{1+\frac{1}{(\log n)^{1+\epsilon}}}.\nonumber
\end{align}
Similarly, for the upper bound, we obtain
\begin{align}
\sum \limits_{\substack{1\leq j\leq n\\x_j\in \mathbb{R}^2\\\mathcal{G}\circ\mathbb{V}_1[\vec{x}_j]=1}}1&\leq  \sum \limits_{\substack{1\leq j\leq n\\x_j\in \mathcal{B}_{\frac{1}{2}\mathcal{G}\circ\mathbb{V}_1[\vec{x}_j]}[\vec{x}_j]\bigcap \mathbb{R}^2\\\mathrm{max}_{1\leq j\leq n}\mathrm{sup}_{1\leq s\leq 2}(x_{j_s})<n^{\frac{1}{(\log\log n)^{\epsilon}}}}}\mathcal{G}\circ \mathbb{V}_1[\vec{x}_j]\nonumber \\& \ll_\epsilon \sum\limits_{\substack{1\leq j\leq n\\ \mathrm{max}_{1\leq j\leq n}\mathrm{sup}_{1\leq s\leq 2}(x_{j_s})<n^{\frac{1}{(\log\log n)^{\epsilon}}}}}\mathrm{sup}_{1\leq s\leq 2}(x_{j_s})\nonumber \\& \ll_\epsilon \sum \limits_{\substack{1\leq j\leq n \\ \mathrm{max}_{1\leq j\leq n}\mathrm{sup}_{1\leq s\leq 2}(x_{j_s})<n^{\frac{1}{(\log\log n)^{\epsilon}}}}}\mathrm{max}_{1\leq j\leq n}\mathrm{sup}_{1\leq s\leq 2}(x_{j_s})\nonumber \\&<n^{\frac{1}{(\log\log n)^{\epsilon}}}\sum\limits_{1\leq j\leq n}1\nonumber \\&= n^{1+\frac{1}{(\log \log n)^{\epsilon}}}.\nonumber
\end{align}
Now for any set of $n$ points in general position in the plane $\mathbb{R}^2$, we apply the translation with a fixed vector $\vec{v}\in \mathbb{R}^2$
\begin{align}
\mathbb{T}_{\vec{v}}:\mathcal{B}_{\frac{1}{2}\mathcal{G}\circ \mathbb{V}_1[\vec{x}_j]}[\vec{x}_j]\longrightarrow \mathcal{B}^{\vec{v}}_{\frac{1}{2}\mathcal{G}\circ \mathbb{V}_1[\vec{x}_j]}[\vec{x}_j]\nonumber
\end{align}
so that the new ball $\mathcal{B}^{\vec{v}}_{\frac{1}{2}\mathcal{G}\circ\mathbb{V}_1[\vec{x}_j]}[\vec{x}_j]$ now lives in the smallest region containing all the $n$ points in general position. We note that this new ball is still of radius $\frac{1}{2}$ but contains points--including admissible points--all of which are translations of the points in the previous ball $\mathcal{B}_{\frac{1}{2}\mathcal{G}\circ\mathbb{V}_1[\vec{x}_j]}[\vec{x}_j]$ by a fixed vector $\vec{v}\in\mathbb{R}^2$. We remark that the unit distances are all preserved so that the number of unit distances determined by the $n$ points in general position is upper bounded by 
\begin{align}
n^{1+\frac{1}{(\log n)^{1+\epsilon}}}\ll_\epsilon \#\mathcal{I}\ll_\epsilon n^{1+\frac{1}{(\log\log n)^{\epsilon}}}.\nonumber
\end{align}
\end{proof}

\subsection{Application to counting integral points in a circle and a grid}

The Gauss circle problem is a problem that seeks to count the number of integral points in a circle centered at the origin and of radius $r$. It is fairly easy to see that the area of a circle with radius $r>0$ gives a fairly good approximation for the number of such integral points in the circle, since on average each unit square in the circle contains at least an integral point. Denoting $N(r)$ to be the number of integral points in a circle of radius $r$, the following elementary estimate is well-known 
\begin{align}
N(r)=\pi r^2+|E(r)|\nonumber
\end{align}
where $|E(r)|$ is the error term. The real and main problem in this area is to obtain a reasonably good estimate of the error term. In fact, it is conjectured that 
\begin{align}
|E(r)|\ll r^{\frac{1}{2}+\epsilon}\nonumber
\end{align}
for $\epsilon>0$. The first fundamental progress was made by Gauss \cite{hardy1999ramanujan}, where it was shown that 
\begin{align}
|E(r)|\leq 2\pi r\sqrt{2}.\nonumber
\end{align}
G.H Hardy and Edmund Landau almost independently obtained a lower bound \cite{hardy1915expression} showing that 
\begin{align}
|E(r)|\neq o(r^{\frac{1}{2}}(\log r)^{\frac{1}{4}}).\nonumber
\end{align}
The current best upper bound (see \cite{huxley2002integer}) is 
\begin{align}
|E(r)|\ll r^{\frac{131}{208}}.\nonumber
\end{align}
and due to Huxley. In this section, we will study a variant of this problem in the region between a general $k$ dimensional grid $2r\times 2r \cdots \times 2r~(k~times)$ and the largest sphere contained in the grid. In particular, we obtain the following lower bound for the number of integral points in this region:

\begin{theorem}
Let $N_k(r)$ denote the number of integral points in the $k$ dimensional sphere of radius $r>0$. We have
\begin{align}
N_k(r)\gg \sqrt{k}\times r^{k-o(1)}.\nonumber
\end{align}
\end{theorem}

\begin{proof}
We arbitrarily choose a point $(x_1,x_2,\ldots,x_k)=\vec{x}\in \mathbb{R}^k$ with $x_i>1$ for $1\leq i\leq k$ and $x_i\neq x_j$ for $i\neq j$ such that $\mathcal{G}\circ \mathbb{V}_m[\vec{x}]=2r$. This ensures that the ball induced under compression has a radius $r$. We apply the compression $\mathbb{V}_m[\vec{x}]$ of the fixed scale $m$ with $0<m\leq 1$ and set $m=m(k)=o(1)$ as $k\longrightarrow \infty$ and construct the ball induced by compression
\begin{align}
\mathcal{B}_{\frac{1}{2}\mathcal{G}\circ \mathbb{V}_m[\vec{x}]}[\vec{x}]\nonumber
\end{align}
with $\frac{(\mathcal{G}\circ \mathbb{V}_m[\vec{x}])}{2}=r$. By Theorem \ref{admissibletheorem} admissible points $\vec{x}_l\in \mathbb{R}^k$~($\vec{x}_l\neq \vec{x}$)~of the ball of compression induced with $||\vec{x}_l-\vec{x}||<\epsilon$ for sufficiently small $\epsilon>0$ must satisfy $\mathcal{G}\circ \mathbb{V}_m[\vec{x}_l]=2r$. Also, by Theorem \ref{decider} points $\vec{x}_l\in \mathcal{B}_{\frac{1}{2}\mathcal{G}\circ \mathbb{V}_m[\vec{x}]}[\vec{x}]$ that are not admissible must satisfy the inequality 
\begin{align}
\mathcal{G}\circ\mathbb{V}_m[\vec{x}_l]\leq \mathcal{G}\circ \mathbb{V}_m[\vec{x}]=2r\nonumber
\end{align}
with $||\vec{x}_l-\vec{x}||<\delta$ for some $\delta>0$. For points $\vec{x}_l\in \mathcal{B}_{\frac{1}{2}\mathcal{G}\circ \mathbb{V}_m[\vec{x}]}[\vec{x}]$ contained in the $2r\times 2r \times\cdots\times 2r~(k~times)$ box that covers this ball, we set
\begin{align}
\mathrm{max}_{\vec{x}_l\in (2r)^k}\mathrm{sup}(x_{l_i})_{i=1}^{k}=\mathrm{min}_{\vec{x}_l\in (2r)^k}\mathrm{inf}(x_{l_i})_{i=1}^{k}=(2r)^{1-o(1)}\nonumber
\end{align}
as $r\longrightarrow \infty$. This ensures that the points in the $k$-dimensional box are also in the ball. For the number of integral points in the largest ball contained in the $2r\times2r\times\cdots\times 2r~(k~times)$ dimensional box, we deduce
\begin{align}
N_k(r)&\geq \sum\limits_{\substack{\vec{x}_l\in (\lfloor 2r \rfloor)^k \subset \mathbb{N}^k\\\mathcal{G}\circ \mathbb{V}_m[\vec{x}_l]\leq 2r}}1\nonumber \\&\geq \sum \limits_{\substack{\vec{x}_l\in (\lfloor 2r \rfloor)^k \subset \mathbb{N}^k\\\mathrm{max}_{\vec{x}_l\in (2r)^k}\mathrm{sup}(x_{l_i})_{i=1}^{k}=(2r)^{1-o(1)}}}\frac{\mathcal{G}\circ\mathbb{V}_m[\vec{x}_l]}{2r}\nonumber \\&\gg \sum \limits_{\substack{\vec{x}_l \in (\lfloor 2r \rfloor)^k\subset \mathbb{N}^k\\1\leq i\leq k\\ \mathrm{max}_{\vec{x}_l\in (2r)^k}\mathrm{sup}(x_{l_i})_{i=1}^{k}=(2r)^{1-o(1)}}}\frac{\sqrt{k}\inf(x_{l_i})}{2r}\nonumber \\&=\frac{1}{2r}\sum \limits_{\substack{\vec{x}_l\in (\lfloor 2r \rfloor)^k\subset \mathbb{N}^k\\1\leq i\leq k\\ \mathrm{max}_{\vec{x}_l\in (2r)^k}\mathrm{sup}(x_{l_i})_{i=1}^{k}=(2r)^{1-o(1)}}}\sqrt{k}\inf (x_{l_i})\nonumber \\&\geq \frac{\sqrt{k}}{2r}\sum \limits_{\substack{\vec{x}_l\in (\lfloor 2r \rfloor)^k\subset \mathbb{N}^k\\1\leq i\leq k\\ \mathrm{max}_{\vec{x}_l\in (2r)^k}\mathrm{sup}(x_{l_i})_{i=1}^{k}=(2r)^{1-o(1)}}}\mathrm{min}_{\vec{x}_l\in (\lfloor 2r \rfloor)^k}\inf(x_{l_i})\nonumber \\&\gg \frac{(2r)^{1-o(1)} \times\sqrt{k}}{2r}\sum \limits_{\substack{\vec{x}_l\in (\lfloor 2r \rfloor)^k \subset \mathbb{N}^k \\1\leq i\leq k}}1\nonumber \\& \gg \frac{(2r)^{1-o(1)}\times \sqrt{k}}{2r}\times r^k \nonumber
\end{align}
where we have used
\begin{align}
\mathrm{max}_{\vec{x}_l\in (2r)^k}\mathrm{sup}(x_{l_i})_{i=1}^{k}=\mathrm{min}_{\vec{x}_l\in (2r)^k}\mathrm{inf}(x_{l_i})_{i=1}^{k}.\nonumber
\end{align}
\end{proof}

\subsection{Application to counting the number of integral points on the boundary of a k-dimensional sphere}

\begin{theorem}
Let $\mathcal{N}_{r,k}$ denote the number of integral points on the boundary of a $k$-dimensional sphere of radius $r$. We have
\begin{align}
\mathcal{N}_{r,k}\gg r^{k-1}\sqrt{k}.\nonumber
\end{align}
\end{theorem}

\begin{proof}
We arbitrarily choose a point $(x_1,x_2,\ldots,x_k)=\vec{x}\in \mathbb{R}^k$ with $x_i>1$ for $1\leq i\leq k$ and $x_i\neq x_j$ for $i\neq j$ such that $\mathcal{G}\circ \mathbb{V}_m[\vec{x}]=2r$. This choice ensures that the ball induced under compression has a radius $r$. We apply the compression $\mathbb{V}_m[\vec{x}]$ of the fixed scale $m$ with $0<m\leq 1$ and set $m:=m(k)=o(1)$ as $k\longrightarrow \infty$. We now construct the ball induced by compression
\begin{align}
\mathcal{B}_{\frac{1}{2}\mathcal{G}\circ\mathbb{V}_m[\vec{x}]}[\vec{x}]\nonumber
\end{align}
with $\frac{(\mathcal{G}\circ\mathbb{V}_m[\vec{x}])}{2}=r$. We note that this ball is exactly covered by the $k$-dimensional box $2r\times2r\times\cdots\times 2r~(k~times)$. By Theorem \ref{admissibletheorem} admissible points $\vec{x}_l\in \mathbb{R}^k$~($\vec{x}_l\neq \vec{x}$)~of the induced compression ball must satisfy the condition $\mathcal{G}\circ \mathbb{V}_m[\vec{x}_l]=2r$. For any small $\epsilon>0$, we set
\begin{align}
\mathrm{max}_{\vec{x}_l\in (2r)^k}\mathrm{sup}(x_{l_i})_{i=1}^{k}<(2r)^{1+\frac{1}{(\log r)^{1+\epsilon}}}
\end{align}
and 
\begin{align}
    \mathrm{min}_{\vec{x}_l\in (2r)^k}\mathrm{inf}(x_{l_i})_{i=1}^{k}=1.
\end{align}
This ensures that 
\begin{align}
    1\leq \mathcal{G}\circ\mathbb{V}_m[\vec{x}]:=2r<(2r)^{1+\frac{1}{(\log r)^{1+\epsilon}}}.
\end{align}
For the number of integral points on the boundary of the $k$-dimensional sphere, we deduce the lower bound
\begin{align}
\mathcal{N}_{r,k}&=\sum \limits_{\substack{\vec{x}_l\in\lfloor 2r \rfloor^k\subset\mathbb{N}^k\\\mathcal{G}\circ \mathbb{V}_m[\vec{x}_l]=2r}}1\nonumber \\& \geq\sum \limits_{\substack{\vec{x}_l\in\lfloor2r\rfloor^k\subset \mathbb{N}^k\\\mathrm{min}_{\vec{x}_l\in (2r)^k}\mathrm{inf}(x_{l_i})_{i=1}^{k}=1\\\mathrm{max}_{\vec{x}_l\in (2r)^k}\mathrm{sup}(x_{l_i})_{i=1}^{k}=2r+o(1)}}\frac{\mathcal{G}\circ \mathbb{V}_m[\vec{x}_l]}{2r}\nonumber \\& \gg \sum \limits_{\substack{\vec{x}_l\in\lfloor2r\rfloor^k\subset \mathbb{N}^k\\1\leq i\leq k\\\mathrm{min}_{\vec{x}_l\in (2r)^k}\mathrm{inf}(x_{l_i})_{i=1}^{k}=1}}\frac{\sqrt{k}\inf (x_{l_i})}{2r}\nonumber \\&\geq \frac{\sqrt{k}}{2r}\sum \limits_{\substack{\vec{x}_l\in \lfloor2r\rfloor^k\subset\mathbb{N}^k\\1\leq i\leq k\\\mathrm{min}_{\vec{x}_l\in (2r)^k}\mathrm{inf}(x_{l_i})_{i=1}^{k}=1}}\mathrm{min}_{\vec{x}_l\in (2r)^k}\mathrm{inf}(x_{l_i})_{i=1}^{k}\nonumber \\&=\frac{\sqrt{k}}{2r}\times\lfloor 2r\rfloor^k\nonumber
\end{align}
and the lower bound follows.
\end{proof}

\subsection{Application to the general distance problem in $\mathbb{R}^k$}

\begin{theorem}
 Let
$$
\mathcal{D}_{n,d,k}=\#\bigg\{(\vec{x}_t,\vec{x_j})\in \mathbb{E}\subset\mathbb{R}^k~:~||\vec{x_j}-\vec{x_t}||=d,~1\leq t,j\leq n\bigg\}.
$$ 
We have
\begin{align}
\mathcal{D}_{n,d,k}\gg \frac{n\sqrt{k}}{d}.\nonumber
\end{align}
\end{theorem}

\begin{proof}
We arbitrarily choose a point $(x_1,x_2,\ldots,x_k)=\vec{x}\in \mathbb{R}^k$ with $x_i>1$ for $1\leq i\leq k$ and $x_i\neq x_j$ for $i\neq j$ such that $\mathcal{G}\circ\mathbb{V}_m[\vec{x}]=d$ for a fixed $d>0$. This ensures that the ball induced under compression has a radius $\frac{d}{2}$. We apply the compression $\mathbb{V}_m[\vec{x}]$ of the fixed scale $m$ with $0<m\leq 1$ and set $m:=m(k)=o(1)$. We construct the ball induced by compression
\begin{align}
\mathcal{B}_{\frac{1}{2}\mathcal{G}\circ\mathbb{V}_m[\vec{x}]}[\vec{x}]\nonumber
\end{align}
with $\frac{(\mathcal{G}\circ \mathbb{V}_m[\vec{x}])}{2}=\frac{d}{2}$. By Theorem \ref{admissibletheorem} the admissible points $\vec{x}_l\in \mathbb{R}^k$~($\vec{x}_l\neq \vec{x}$)~of the induced compression ball must satisfy the condition $\mathcal{G}\circ\mathbb{V}_m[\vec{x}_l]=d$. We set 
$$
\mathrm{max}_{1\leq l\leq n}\mathrm{sup}_{1\leq i\leq k}(x_{l_i})=d+o(1)\quad (n\longrightarrow \infty)
$$
and 
$$
\mathrm{min}_{1\leq l\leq n}\mathrm{inf}_{1\leq i\leq k}(x_{j_s})>1.
$$
We now count the number of $d$-unit distances formed by a set of $n$ points in $\mathbb{R}^k$ by counting pairs of admissible points $(\vec{x}_l,\vec{x}_h)$ on the ball $\mathcal{B}_{\frac{1}{2}\mathcal{G}\circ \mathbb{V}_m[\vec{x}]}[\vec{x}]$ such that $\mathbb{V}_m[\vec{x}_l]=\vec{x}_h$. For the lower bound of the number of $d$-unit distances, we deduce 
\begin{align}
\mathcal{D}_{n,d,k}&\geq \sum \limits_{\substack{1\leq l\leq \frac{n}{2}\\\vec{x}_l\in \mathbb{R}^k\\\mathcal{G}\circ \mathbb{V}_m[\vec{x}_l]=d}}1\nonumber \\&=\sum \limits_{\substack{1\leq l\leq \frac{n}{2}\\\vec{x}_l\in \mathbb{R}^k\\\mathrm{max}_{1\leq l\leq n}\mathrm{sup}_{1\leq i\leq k}(x_{l_i})=d+o(1)\\\mathrm{min}_{1\leq l\leq n}\mathrm{inf}_{1\leq i\leq k}(x_{j_s})>1.}}\frac{\mathcal{G}\circ\mathbb{V}_m[\vec{x}_l]}{d}\nonumber \\& \gg \sum \limits_{\substack{1\leq l\leq \frac{n}{2}\\1\leq i\leq k\\\mathrm{min}_{1\leq l\leq n}\mathrm{inf}_{1\leq i\leq k}(x_{j_s})>1.}}\frac{\sqrt{k}\mathrm{min}_{1\leq l\leq n}\mathrm{inf}(x_{l_i})}{d}\nonumber \\&\geq \frac{\sqrt{k}}{d}\sum \limits_{1\leq l\leq \frac{n}{2}}1\nonumber \\&=\frac{n\sqrt{k}}{2d}\nonumber
\end{align}
and the lower bound follows.
\end{proof}

\subsection{Application to counting the average number of integer powered distances in $\mathbb{R}^k$}

\begin{theorem}
 Let
$$
\mathcal{D}_{n,d^r,k}=\#\bigg\{(\vec{x}_t,\vec{x_j})\in \mathbb{E}\subset\mathbb{R}^k~:~||\vec{x_j}-\vec{x_t}||=d^r,~1\leq t,j\leq n,r\geq 1\bigg\}.
$$ 
We have
\begin{align}
\sum\limits_{1\leq d\leq t}\mathcal{D}_{n,k,d^r}\gg n\sqrt[2r]{k}\log t\nonumber
\end{align}
for a fixed $t>1$.
\end{theorem}

\begin{proof}
We arbitrarily choose a point $(x_1,x_2,\ldots,x_k)=\vec{x}\in \mathbb{R}^k$ with $x_i>1$ for $1\leq i\leq k$ and $x_i\neq x_j$ for $i\neq j$ such that $\mathcal{G}\circ \mathbb{V}_m[\vec{x}]=d^r$ for fixed $d>0$ and $r>1$. This  ensures that the induced compression ball has a radius $\frac{d^r}{2}$. We apply the compression $\mathbb{V}_m[\vec{x}]$ of the fixed scale $m$ with $0<m\leq 1$ and set $m:=m(k)=o(1)$ as $k\longrightarrow \infty$. We construct the ball induced by compression
\begin{align}
\mathcal{B}_{\frac{1}{2}\mathcal{G}\circ\mathbb{V}_m[\vec{x}]}[\vec{x}]\nonumber
\end{align}
with $\frac{(\mathcal{G}\circ\mathbb{V}_m[\vec{x}])}{2}=\frac{d^r}{2}$. By Theorem \ref{admissibletheorem} the admissible points $\vec{x}_l\in \mathbb{R}^k$~($\vec{x}_l\neq \vec{x}$)~of the induced compression ball must satisfy the condition $\mathcal{G}\circ\mathbb{V}_m[\vec{x}_l]=d^r$. We set 
$$
\mathrm{max}_{1\leq l\leq n}\mathrm{sup}_{1\leq i\leq k}(x_{l_i})=d^r+o(1)\quad (n\longrightarrow \infty)
$$
and 
$$
\mathrm{min}_{1\leq l\leq n}\mathrm{inf}_{1\leq i\leq k}(x_{l_i})>1.
$$
We now count the number of $d^r$-unit distances formed by a set of $n$ points in $\mathbb{R}^k$ by counting pairs of admissible points $(\vec{x}_l,\vec{x}_h)$ on the ball $\mathcal{B}_{\frac{1}{2}\mathcal{G}\circ\mathbb{V}_m[\vec{x}]}[\vec{x}]$ such that $\mathbb{V}_m[\vec{x}_l]=\vec{x}_h$. For the average number of $d^r$-unit distances for $1\leq d\leq t$ with fixed $t,r>1$, we deduce 
\begin{align}
\sum\limits_{1\leq d\leq t}\mathcal{D}_{n,k,d^r}&=\sum \limits_{1\leq d\leq t}\sum\limits_{\substack{1\leq l\leq \frac{n}{2}\\\vec{x}_l\in \mathbb{R}^k\\\mathcal{G}\circ \mathbb{V}_m[\vec{x}_l]=d^r}}1\nonumber \\&=\sum\limits_{1\leq d\leq t}\sum \limits_{\substack{1\leq l\leq \frac{n}{2}\\\vec{x}_l\in \mathbb{R}^k\\\mathrm{max}_{1\leq l\leq n}\mathrm{sup}_{1\leq i\leq k}(x_{l_i})=d^r+o(1)\\\mathrm{min}_{1\leq l\leq n}\mathrm{inf}_{1\leq i\leq k}(x_{l_i})>1}}\frac{\sqrt[r]{\mathcal{G}\circ \mathbb{V}_m[\vec{x}_l]}}{d}\nonumber \\& \gg \sum\limits_{1\leq d\leq t}\sum \limits_{\substack{1\leq l\leq \frac{n}{2}\\1\leq i\leq k\\\mathrm{min}_{1\leq l\leq n}\mathrm{inf}_{1\leq i\leq k}(x_{l_i})>1}}\frac{\sqrt[2r]{k}\sqrt[r]{\mathrm{min}_{1\leq l\leq n}\mathrm{inf}(x_{l_i})}}{d}\nonumber \\&\geq \sum \limits_{1\leq d\leq t}\frac{\sqrt[2r]{k}}{d}\sum \limits_{1\leq l\leq \frac{n}{2}}1\nonumber \\&=\sum \limits_{1\leq d\leq t}\frac{n\sqrt[2r]{k}}{2d}\nonumber \\&=\frac{n\sqrt[2r]{k}}{2}\sum \limits_{1\leq d\leq t}\frac{1}{d}\nonumber
\end{align}
and the lower bound follows.
\end{proof}

\subsection{Application to the Ehrhart volume conjecture}

The Ehrhart volume conjecture is the assertion that any convex body $K$ in $\mathbb{R}^n$ with a single lattice point in it's interior as barycenter must have volume satisfying the upper bound
\begin{align}
    Vol(K)\leq \frac{(n+1)^n}{n!}.\nonumber
\end{align}
The conjecture has only been proven for various special cases in very specific settings. For instance, Ehrhart proved the conjecture in the two dimensional case and for simplices \cite{ehrhart1979volume}. The conjecture has also been settled for a large class of rational polytopes \cite{berman2017volume}. In this section, we study the Ehrhart volume conjecture. We will show that the claimed inequality fails for some convex bodies, thereby providing a counter-example to the Ehrhart volume conjecture. The main idea is a construction of a ball in $\mathbb{R}^n$ and the realization that after a slight adjustment of the internal structure, the ball satisfies the requirements of the conjecture but has too much volume--at least a volume beyond that postulated by Ehrhart. In particular, we prove the following lower bound

\begin{theorem}
Let $Vol(K)$ denotes the volume of a ball in $\mathbb{R}^n$ with only one lattice points in it's interior as its center of mass. We have
\begin{align}
Vol(K) \gg \frac{n^n}{\sqrt{n}}.\nonumber
\end{align}
\end{theorem}

\begin{proof}
We arbitrarily choose a point $(x_1,x_2,\ldots,x_n)=\vec{x}\in \mathbb{R}^n$ with $x_i>1$ for $1\leq i\leq n$ and $x_i\neq x_j$ for $i\neq j$ such that $\mathcal{G}\circ \mathbb{V}_m[\vec{x}]=n$. This ensures that the induced compression ball has radius $\frac{n}{2}$. We apply the compression $\mathbb{V}_m[\vec{x}]$ of fixed scale $m$ with $0<m\leq 1$ and set $m:=m(n)=o(1)$ as $n\longrightarrow \infty$. We now construct the ball induced by compression
\begin{align}
K:=\mathcal{B}_{\frac{1}{2}\mathcal{G}\circ\mathbb{V}_m[\vec{x}]}[\vec{x}]\nonumber
\end{align}
with $\frac{(\mathcal{G}\circ\mathbb{V}_m[\vec{x}])}{2}=\frac{n}{2}$. By Theorem \ref{admissibletheorem} the admissible points $\vec{x}_l\in \mathbb{R}^k$~($\vec{x}_l \neq \vec{x}$)~of the induced compression ball must satisfy the condition $\mathcal{G}\circ\mathbb{V}_m[\vec{x}_l]=n$ with $||\vec{x}_l-\vec{x}||<\delta$ for sufficiently small $\delta>0$. Also, by Theorem \ref{decider}, the points $\vec{x}_l \in \mathcal{B}_{\frac{1}{2}\mathcal{G}\circ\mathbb{V}_m[\vec{x}]}[\vec{x}]$ must satisfy the inequality 
\begin{align}
\mathcal{G}\circ \mathbb{V}_m[\vec{x}_l]<\mathcal{G}\circ \mathbb{V}_m[\vec{x}]=n.\nonumber
\end{align} 
We set 
$$
\mathrm{max}_{\vec{x}_l\in n^n}\mathrm{sup}(x_{l_i})_{i=1}^{n}\leq n+o(1)
$$
and 
$$
\mathrm{min}_{\vec{x}_l\in n^n}\mathrm{inf}(x_{l_i})_{i=1}^{n}>1.
$$
For the number of integral points in the largest ball contained in the $n\times n\times \cdots \times n~(n~times)$ grid that shares admissible points on both sides with the grid is 
\begin{align}
N_n(n)&=\sum\limits_{\substack{\vec{x}_l\in n^n\subset \mathbb{R}^n\\\mathcal{G}\circ\mathbb{V}_m[\vec{x}_l]\leq n}}1\nonumber \\&\geq \sum\limits_{\substack{\vec{x}_l\in n^n \subset \mathbb{R}^n\\\mathrm{min}_{\vec{x}_l\in n^n}\mathrm{inf}(x_{l_i})_{i=1}^{n}>1\\\mathrm{max}_{\vec{x}_l\in n^n}\mathrm{sup}(x_{l_i})_{i=1}^{n}\leq n+o(1)}}\frac{\mathcal{G}\circ\mathbb{V}_m[\vec{x}_l]}{n}\nonumber \\& \gg \frac{1}{n}\sum\limits_{\substack{\vec{x}_l\in n^n\subset \mathbb{R}^n\\ 1\leq i\leq n\\\mathrm{max}_{\vec{x}_l\in n^n}\mathrm{sup}(x_{l_i})_{i=1}^{n}\leq n+o(1)\\\mathrm{min}_{\vec{x}_l\in n^n}\mathrm{inf}(x_{l_i})_{i=1}^{n}>1}}\sqrt{n}\inf (x_{l_i})\nonumber \\&\geq \frac{\sqrt{n}}{n}\sum \limits_{\substack{\vec{x}_l\in n^n \subset \mathbb{R}^n\\1\leq i\leq n\\\mathrm{min}_{\vec{x}_l\in n^n}\mathrm{inf}(x_{l_i})_{i=1}^{n}>1}}\mathrm{min}_{\vec{x}_l\in n^n}\inf(x_{l_i})\nonumber \\&\gg \frac{\mathrm{min}_{\vec{x}_l\in n^n}\mathrm{inf}(x_{l_i})_{i=1}^{n} \times \sqrt{n}}{n}\sum \limits_{\substack{\vec{x}_l\in n^n\subset \mathbb{R}^n \\1\leq i\leq n}}1\nonumber \\& \gg \frac{\sqrt{n}}{n}\times n^n.\nonumber
\end{align}
The number of lattice points $N_n(n)$ in the ball  $K:=\mathcal{B}_{\frac{1}{2}\mathcal{G}\circ\mathbb{V}_m[\vec{x}]}[\vec{x}]$ and the volume $Vol(K)$ satisfies the asymptotic relation $N_n(n)\sim Vol(K)$ so that by removing all sub-grid of the grid $n\times n\cdots \times n$~($n$~times) contained in the ball $K:=\mathcal{B}_{\frac{1}{2}\mathcal{G}\circ \mathbb{V}_m[\vec{x}]}[\vec{x}]$ except the sub-grid $\frac{n}{2}\times \frac{n}{2}\times \cdots \frac{n}{2}$~($n$~times), we see that we are left with only one lattice point as the center of the ball. This completes the construction.
\end{proof}

\subsection{Application to counting the maximum number of points in a plane figure with large pairwise distances}

Let $d>0$. The following question appears in \cite{smarandache1991only} 

\begin{question}\label{problem}
What is the maximum number of points included in a plane figure (generally: in a space body) such that the distance between any two points is greater than or equal to $d$?
\end{question}

Although this question belongs to the class of discrete geometry problems involving certain configurations of points and lines in the plane (resp. Euclidean space), Problem \ref{problem} is relatively unknown and unsolved. Depending on the dimension of the space in which the points dwell, the problem demands a precise arrangement of points so that their mutual distances are not small and are totally covered by a planar figure (resp. space body). In fact, the problem might be investigated by selecting a planar (resp. space curve) that contains all of these points in the correct configuration, as this curve can be embedded in a planar shape (resp. space body) or its slightly expanded and translated equivalents. This is the main concept we will use in our investigation. Using the \emph{compression} geometry developed, we show that the maximum number of points that can be included in a planar figure with mutual distances of at least $d>0$ is at least $d^{\epsilon}$ for a small $\epsilon>0$. In particular, we obtain the following lower bound:

\begin{theorem}
Let $\Delta_2(d)$ denote the maximum number of points that can be placed within a geometric figure in $\mathbb{R}^2$ such that their mutual distances are at least $d>0$. We have
\begin{align}
\Delta_2(d)\gg d^{\epsilon}\nonumber
\end{align}
for some small $\epsilon>0$.
\end{theorem}

Here, we use an equivalent notion of the circumference of the circle induced by points under compression in the plane $\mathbb{R}^2$ as follows:
\bigskip

\begin{proposition}\label{circumference}
Let $\vec{x}\in \mathbb{R}^2$ with $x_i\neq 0$ for each $1\leq i\leq 2$. Let $\delta(\mathbb{V}_m[\vec{x}])$ denote the circumference of the circle induced by the point $\vec{x}$ under compression $\mathbb{V}_m[\vec{x}]$ of a fixed scale $m$ with $0<m\leq 1$. We have
\begin{align}
\delta(\mathbb{V}_m[\vec{x}])=\pi \times (\mathcal{G}\circ \mathbb{V}_m[\vec{x}]).\nonumber
\end{align}
\end{proposition}

\begin{proof}
This follows from the standard definition of the circumference of a circle and from observing that the radius $r$ of the circle induced by the point $\vec{x}\in \mathbb{R}^2$ under compression is 
\begin{align}
r=\frac{\mathcal{G}\circ\mathbb{V}_m[\vec{x}]}{2}.\nonumber
\end{align}
\end{proof}

\begin{theorem}
Let $\Delta_2(d)$ denote the maximum number of points that can be placed within a geometric figure in $\mathbb{R}^2$ such that their mutual distances are at least $d>0$. For any $\epsilon>0$, we have
\begin{align}
\Delta_2(d)\gg_\epsilon d^{\epsilon}.\nonumber
\end{align}
\end{theorem}

\begin{proof}
We arbitrarily choose a point $(x_1,x_2)=\vec{x}\in\mathbb{R}^2$ such that $\mathcal{G}\circ\mathbb{V}_m[\vec{x}]\geq d^{f(d)}$. We apply the compression $\mathbb{V}_m[\vec{x}]$ of the fixed scale $m$ with $1\geq m>0$ to the point and set $m:=m(2)=\frac{1}{2}$. We now construct the circle induced by compression 
\begin{align}
\mathcal{B}_{\frac{1}{2}\mathcal{G}\circ\mathbb{V}_m[\vec{x}]}[\vec{x}]\nonumber
\end{align}
with $\frac{(\mathcal{G}\circ \mathbb{V}_m[\vec{x}])}{2}\geq \frac{d^{f(d)}}{2}$ by choosing 
\begin{align}
\mathrm{sup}(x_i)_{1\leq i\leq 2}=\mathrm{inf}(x_i)_{1\leq i\leq 2}=d^{f(d)+\epsilon}\nonumber
\end{align}
for any small $\epsilon>0$ and for some function $f:\mathbb{R}\longrightarrow \mathbb{R}$ such that the constructed circle of compression lives in the plane figure. On this circle, we locate admissible points so that the chord that joins each pair of adjacent admissible points is of length $d>0$. By Proposition \ref{circumference}, we obtain the circumference of the circle induced by compression
\begin{align}
\delta(\mathbb{V}_m[\vec{x}])=\pi \times \mathcal{G}\circ \mathbb{V}_m[\vec{x}].\nonumber
\end{align}
We join all pairs of adjacent admissible points considered by a chord. We note that we can use the length of the arc induced by any two adjacent admissible points on the circle to determine the number of pairwise admissible points with mutual distances at least $d>0$. We deduce that the number of admissible points on the circle with mutual distances at least $d>0$ satisfies
\begin{align}
\Delta_2(d):&=\frac{\pi \times (\mathcal{G}\circ \mathbb{V}_m[\vec{x}])}{\frac{d^{f(d)}}{2}\theta}\nonumber \\&\gg \frac{\mathrm{inf}(x_i)_{1\leq i\leq 2}}{\frac{d^{f(d)}}{2}}\nonumber \\&=d^{\epsilon}\nonumber
\end{align}
for a fixed $0<\theta:=\theta(d)\leq \pi$. This completes the proof of the lower bound.
\end{proof}

\section{Further remarks}

The method of compression could be a potentially useful tool for resolving the Erd\'{o}s-Straus conjecture. It can also find its place as a tool-box for a good number of Diophantine problems. It may also be seen as a tool-kit for approaching problems in discrete and combinatorial geometry.

%%%%%%%%%%%%%%%%%%%%%%%%%%%%%%%%%%%%%%%%%%%%%%%%%%%%%%%%%%%%%%%%%%%%%%%%
\footnote{
\par
.}%
%%%%%%%%%%%%%%%%%%%%%%%%%%%%%%%%%%%%%%%%%%%%%%%%%%%%%%%%%%%%%%%%%%%%%%%%

\bibliographystyle{amsplain}

\end{document}